\documentclass[10pt,colorinlistoftodos]{amsart}
\usepackage{lmodern}
\usepackage[utf8]{inputenc}
\usepackage[T1]{fontenc}
\usepackage[english]{babel}
\usepackage{csquotes}
\usepackage{comment}

\usepackage[backend=biber,style=alphabetic,url=false,giveninits=true,sorting=anyt,date=year]{biblatex}
\usepackage{amsmath,amssymb,amsthm}
\usepackage{amsaddr}
\usepackage{mathtools,thmtools,stmaryrd,esint}
\usepackage{mathrsfs}
\usepackage{hyperref}
\usepackage[shortlabels]{enumitem}
\usepackage[nameinlink]{cleveref}

\usepackage{longtable,tabu}
\usepackage{aligned-overset}

\usepackage[disable]{todonotes}

\numberwithin{equation}{section}
\hypersetup{colorlinks=true}
\makeatletter
\@namedef{subjclassname@2020}{\textup{2020} Mathematics Subject Classification}
\makeatother

\declaretheorem[name=Theorem,within=section]{theorem}
\declaretheorem[name=Lemma,numberlike=theorem]{lemma}
\declaretheorem[name=Proposition,numberlike=theorem]{proposition}
\declaretheorem[name=Corollary,numberlike=theorem]{corollary}
\declaretheorem[name=Definition,numberlike=theorem]{definition}
\declaretheorem[name=Remark,numberlike=theorem,style=remark]{remark}
\declaretheorem[name=Example,numberlike=theorem,style=remark]{example}

\newcommand{\dd}{\mathop{}\mathopen{}\mathrm{d}}

\newcommand{\N}{\mathbb{N}}

\newcommand{\R}{\mathbb{R}}
\newcommand{\Z}{\mathbb{Z}}
\newcommand{\Sph}{\mathbb{S}}
\newcommand{\Cspace}{\mathscr{C}}
\newcommand{\Mspace}{\mathscr{M}}
\newcommand{\ener}{\mathcal{E}}
\newcommand{\mass}{\mathbf{M}}
\newcommand{\mres}{\mathop{\hbox{\vrule height 6pt width .5pt depth 0pt \vrule height .5pt width 4pt depth 0pt}}\nolimits}
\newcommand{\id}{\mathrm{Id}}
\newcommand{\loc}{\mathrm{loc}}
\newcommand{\lsc}{\mathrm{lsc}}
\newcommand{\eps}{\varepsilon}
\newcommand{\uloc}{L^1_\mathrm{uloc}}
\newcommand{\st}{\;:\;}

\DeclarePairedDelimiter{\abs}{\lvert}{\rvert}
\DeclarePairedDelimiter{\norm}{\lVert}{\rVert}
\DeclarePairedDelimiterX{\intint}[2]{\llbracket}{\rrbracket}{#1,#2}
\DeclareMathOperator{\supp}{supp}
\DeclareMathOperator{\diam}{diam}

\DeclareMathOperator*{\esssup}{ess\,sup}
\DeclareMathOperator{\hdm}{\mathcal{H}}
\DeclareMathOperator{\lbm}{\mathcal{L}}
\DeclareMathOperator{\dist}{dist}
\DeclareMathOperator{\reach}{reach}

\newcommand{\weakto}{\rightharpoonup}
\newcommand{\mweakto}{\xrightharpoonup{\Cspace_0'}}
\newcommand{\narrowto}{\xrightharpoonup{\Cspace_b'}}
\newcommand{\xto}[2][]{\xrightarrow[#2]{#1}}

\newcommand{\xnarrowto}[1]{\xrightharpoonup[#1]{\Cspace_b'}}

\begin{document}
\title{Mass concentration in rescaled first order integral functionals}
% Titre français
% Concentration de masse dans des fonctionnelles intégrales d'ordre un rééchelonnées
\begin{abstract}
We consider first order local minimization problems of the form $\min \int_{\mathbb{R}^N}f(u,\nabla u)$ under a mass constraint $\int_{\mathbb{R}^N}u=m$. We prove that the minimal energy function $H(m)$ is always concave, and that relevant rescalings of the energy, depending on a small parameter $\varepsilon$, $\Gamma$-converge towards the $H$-mass, defined for atomic measures $\sum_i m_i\delta_{x_i}$ as $\sum_i H(m_i)$. We also consider Lagrangians depending on $\varepsilon$, as well as space-inhomogeneous Lagrangians and $H$-masses. Our result holds under mild assumptions on $f$, and covers in particular $\alpha$-masses in any dimension $N\geq 2$ for exponents $\alpha$ above a critical threshold, and all concave $H$-masses in dimension $N=1$. Our result yields in particular the concentration of Cahn-Hilliard fluids into droplets, and is related to the approximation of branched transport by elliptic energies.

% Résumé français 
%Nous considérons des problèmes de minimisation locaux d'ordre un de la forme $\min \int_{\mathbb{R}^N}f(u,\nabla u)$ sous contrainte de masse $\int_{\mathbb{R}^N}u=m$. Nous prouvons que la fonction d'énergie minimale $H(m)$ est toujours concave, et que des rééchelonnements appropriés de l'énergie, dépendant d'un petit paramètre $\varepsilon$, $\Gamma$-convergent vers la $H$-masse, définie pour les mesures atomiques $\sum_i m_i\delta_{x_i}$ par $\sum_i H(m_i)$. Nous considérons également des Lagrangiens dépendant de $\varepsilon$, et des Lagrangiens et $H$-masses spatialement inhomogènes. Notre résultat est valable sous de faibles hypothèses sur $f$, et couvre les $\alpha$-masses en toute dimension $N\geq 2$ pour des exposants $\alpha$ au-dessus d'un seuil critique, et toutes les $H$-masses concaves en dimension $N=1$. Notre résultat donne en particulier la concentration des fluides de Cahn-Hilliard en gouttelettes, et est lié à l'approximation du transport branché par des énergies elliptiques.
\end{abstract}


% Keywords
% mots clés français
% $\Gamma$-convergence, semi-continuité, fonctionnelles intégrales, convergence des mesures, concentration-compacité, fluides de Cahn-Hilliard, transport branché

% 49J45 %Methods involving semicontinuity and convergence; relaxation
% 28A33 %Spaces of measures, convergence of measures
% 46E35 %Sobolev spaces and other spaces of “smooth”
% functions, embedding theorems, trace theorem
% 49Q20 %Variational problems in a geometric measure-
% theoretic setting
% 49J10 %Existence theories for free problems in two or
% more independent variables
% 49Q22 %Optimal transportation
% 76T99 %None of the above, but in this section ; Multiphase and multicomponent flows ; Fluid mechanics

\author{Antonin Monteil}
\address{Universit\'e Paris-Est Cr\'eteil Val-de-Marne, LAMA, France}
%\curraddr{}
%\email{}
\thanks{A.~M. acknowledges support by Leverhulme grant RPG-2018-438.}

\author{Paul Pegon}
\address{Universit\'e Paris-Dauphine, Ceremade \& INRIA Paris, MOKAPLAN, France}
%\curraddr{}
%\email{}
%\thanks{}

\date\today
\maketitle
\tableofcontents

\section*{Notation}
\begin{tabu}  {X[2,c,m]  X[12,l,p]}
%{X[2,c,m]  X[5,l,p]}
\(B_r(x)\) & open ball of radius \(r\) centered at \(x\);\\
\(B_r\) & open ball $B_r(0)$;\\
\(\Mspace(\R^N)\) & set of finite signed Borel measures on \(\R^N\);\\
 \(\Mspace_+(\R^N)\) & set of finite positive Borel measures on \(\R^N\);\\
$\Phi_\sharp\mu$ & pushforward of a measure $\mu\in\Mspace(\R^N)$ by a map $\Phi:\R^N\to\R^k$, defined as $A\mapsto \mu(\Phi^{-1}(A))$;\\
$\tau_x\mu$ & Borel measure $A \mapsto \mu(A-x)$ if \(\mu\in\Mspace(\R^N)\) and $x\in \R^N$;\\
$c_B \mu$ & Borel measure $\tau_{-x} (\mu \mres B)$ if $B$ is the ball $B_r(x)$;\\
$\mu_\ell \mweakto \mu$ & weak convergence of measures, i.e. weak-$\star$ convergence in duality with the space $\Cspace_0(\R^N)$ of continuous functions vanishing at infinity;\\
$\mu_\ell \narrowto \mu$ & narrow convergence of measures, i.e. weak-$\star$ convergence in duality with the space $\Cspace_b(\R^N)$ of continuous and bounded  function;\\
$\Sigma$ & set of increasing maps $\sigma : \N \to \N$;\\
$\sigma_1 \preceq \sigma_2$ & $\sigma_1, \sigma_2 \in \Sigma$ are such that $\sigma_1(\intint{n}{+\infty}) \subseteq \sigma_2(\N)$ for some $n\in \N$ ;\\
$\pm$ & fixed to either $+$ or $-$ in the whole statement or proof, and $\mp = -(\pm)$.
\end{tabu}

\section{Introduction}\label{sectionIntroduction}

\subsection{Setting}

Let \(N\in\N^\ast\) and let \(f:\R^N\times\R\times\R^N\to[0,+\infty]\) be a Borel function such that $f(\cdot,0,0) \equiv 0$. Consider the following energy functional, defined for any fixed \(x\in\R^N\) on the set of finite Borel measures \(\Mspace(\R^N)\) on \(\R^N\) by
\begin{equation}\label{def_ener}
\ener^x_f(\mu)=
\begin{dcases}
\int_{\R^N} f(x,u(y),\nabla u(y))\dd y &\text{if } \mu = u\lbm^N, u\in W^{1,1}_\loc(\R^N),\\
+\infty&\text{otherwise.}
\end{dcases}
\end{equation}
The minimization of this energy under a mass constraint gives rise to the notion of minimal cost function, valued in $[0,+\infty]$ and defined by
\begin{equation}\label{def_Hf}
H_f(x,m) \coloneqq \inf\left\{\ener^x_f(u\lbm^N) \st u \in W^{1,1}_\loc\cap L^1(\R^N) \text{ such that } \int_{\R^N} u = m\right\}.
\end{equation}

As a preliminary result, which applies to \labelcref{def_ener} and deserves interest on its own, we will establish the following:
\begin{theorem}\label{concaveDN}
Let $f :\R\times \R^N \to [0,+\infty]$ be Borel measurable such that $f(0,0) = 0$. The function defined for every $m\in \R$ by
\begin{equation}\label{PropertiesDefinitionH}
H_f(m) \coloneqq \inf\left\{\int_{\R^N}f(u,\nabla u)\st u\in W^{1,1}_\loc\cap L^1(\R^N,\R),\,\int_{\R^N}u=m\right\}
\end{equation}
vanishes at $0$ and it is either identically $+\infty$ on $(0,+\infty)$, or it is everywhere finite, continuous, concave and non-decreasing on \([0,+\infty)\). The symmetric statement on $(-\infty,0]$ holds as well.
\end{theorem}
\noindent The proof is very simple and works with no further assumptions on \(f\).

Our main purpose is to prove that if $(f_\eps)_{\eps > 0}$ is a family of functions $f_\eps:\R^N\times\R\times\R^N\to[0,+\infty]$ which converges pointwise to \(f\) as \(\eps\to 0\) and satisfies some conditions, then the rescaled energy functionals \(\ener_\eps\), defined for each \(\eps>0\) on \(\Mspace(\R^N)\) by
\begin{equation}
\label{mainEnergyEpsilon}
\ener_\eps(\mu)= 
\begin{dcases}
\int_{\R^N} f_\eps(x,\eps^N u(x), \eps^{N+1} \nabla u(x)) \eps^{-N} \dd x &\text{if }\mu = u\lbm^N, u\in W^{1,1}_\loc(\R^N),\\
+\infty&\text{otherwise,}
\end{dcases}
\end{equation}
\(\Gamma\)-converge as \(\eps\to 0\), for the narrow or weak convergence of measures, to the \(H_f\)-mass, defined on \(\Mspace(\R^N)\) by (see \Cref{defHmass}):
\begin{multline*}
\mass^{H_f}(\mu) \coloneqq \int_{\R^N} H_f(x,\mu(\{x\}))\dd\hdm^0(x)\\
 +\int_{\R^N} H_f'(x,0^+) \dd \mu^d_+(x) + \int_{\R^N} H_f'(x,0^-) \dd \mu^d_-(x).
 \end{multline*}
where \(\mu = \mu^a+\mu^d\) is the decomposition of $\mu$ into its atomic part $\mu^a$ and its diffuse part $\mu^d$, $\mu^d = \mu^d_+ - \mu^d_-$ is the Jordan decomposition of $\mu^d$ into positive and negative parts, and $H_f'(x,0^\pm)$ is the recession at $0$, that is
\[H_f'(x,0^\pm) \coloneqq \lim_{m\to 0^\pm}\frac{H_f(x,m)}{\abs{m}}\in [0,+\infty].\]

Notice that the exponents over $\eps$ in the definition of $\ener_\eps$ are tuned so that if \(B_r(x_0) \subseteq\R^N\) and $u_\eps$ is a mass-preserving rescaling of $v_\eps$ given by $u_\eps(x) \eqqcolon \eps^{-N} v_\eps(\eps^{-1}(x-x_0))$, then
\[
\int_{B_r(x_0)} f_\eps(x,\eps^N u_\eps(x), \eps^{N+1} \nabla u_\eps(x)) \eps^{-N} \dd x  = \int_{B_{r/\eps}} f_\eps(x_0+\eps y,v_\eps(y), \nabla v_\eps(y)) \dd y,
\]
so that the energy contribution of a mass \(m\geq 0\) contained in a ball \(B_r(x_0)\) should be of the order of \(H_f(x_0,m)\). This explains why \(\mass^{H_f}\) is expected to be the \(\Gamma\)-limit of \(\ener_\eps\). Nevertheless, it is not true in general (see \Cref{sectionExamplesIntroduction} below), and we will need a couple of assumptions on \(f\) and \(f_\eps\) detailed in the next section.

This kind of singular limit of integral functionals is reminiscent of several variational models with physical relevance which have been the object of intensive mathematical analysis, such as Cahn-Hilliard fluids with concentration on droplets \cite{bouchitteTransitionsPhasesAvec1996} (which we recover in \Cref{droplets}) or on singular interfaces \cite{modicaEsempioDiGamma1977}, toy models for micromagnetism and liquid crystals like Aviles-Giga \cite{avilesLowerSemicontinuityDefect1999} and Landau-de Gennes \cite{baumanAnalysisNematicLiquid2012}, or Ginzburg-Landau theory of superconductivity \cite{bethuelGinzburgLandauVortices2017}.

\subsection{Assumptions and main result}

Our first two assumptions are rather standard and guarantee the sequential lower semicontinuity of the functionals \(\ener_f^x\),

\begin{enumerate}[\textnormal{(H}$_1$\textnormal{)},leftmargin=*]
\item\label{hypLowerContinuous} \(f : \R^N\times\R\times\R^N \to [0,+\infty]\) is lower semicontinuous,
\end{enumerate}

\begin{enumerate}[\textnormal{(H}$_1$\textnormal{)},leftmargin=*,resume]
\item\label{hypConvex} $f(x,u,\cdot)$ is convex for every $x\in \R^N$, $u\in\R$.\end{enumerate}

In order for vanishing parts to have no energetic contribution, we will impose
\begin{enumerate}[\textnormal{(H}$_1$\textnormal{)},leftmargin=*,resume]
\item\label{hypZero} $f(x,0,0)=0$ for every $x\in \R^N$.
\end{enumerate}

We also need continuity in the spatial variable,
\begin{enumerate}[\textnormal{(H}$_1$\textnormal{)},leftmargin=*,resume]
\item\label{hypContinuous} $f(\cdot,u,\xi)$ is continuous for every $u\in\R,\xi\in \R^N$.
\end{enumerate}

Next, we need a compactness assumption which ensures relative compactness in the weak topology of \(W^{1,p}_\loc(\R^N)\) for sequences of bounded energy \(\ener_f^x\) and bounded mass; it will also be needed in obtaining lower bounds for the energy (see \Cref{vanishingCost}):
\begin{enumerate}[\textnormal{(H}$_1$\textnormal{)},leftmargin=*,resume]
\item\label{hypCompact} there exist \(\alpha,\beta\in(0,+\infty)\), \(p\in (1,+\infty)\) such that for all $(x,u,\xi)\in \R^N \times\R\times\R^N$,
\[f(x,u,\xi) \geq \alpha\abs{\xi}^p-\beta u.\]
\end{enumerate}

We also impose a comparison condition on the slopes of \(f(x,\cdot,\xi)\) and $H_f(x,\cdot)$ at the origin, which will be needed in order to show that the $\Gamma-\liminf$ is bounded from below by the $H_f$-mass, and which rules out some non-trivial scale invariant Lagrangians for which the expected $\Gamma$-convergence result fails (see \Cref{sectionExamplesIntroduction}):
\begin{enumerate}[\textnormal{(H}$_1$\textnormal{)},leftmargin=*,resume]
\item\label{hypSlope} for every $x_0\in \R^N$,
\begin{equation}\label{slopeInfSimplified}
H_f'(x_0,0^\pm) \leq f'_-(x_0,0^\pm,0)\coloneqq\displaystyle\liminf_{(x,u,\xi) \to (x_0,0^\pm,0)} \frac{f(x,u,\xi)}{\abs{u}}.
\end{equation}
\end{enumerate}
We give a general assumption (simple in dimension one but quite technical in dimension $N\geq 2$) depending only on the Lagrangian $f$ so as to guarantee such a condition (see  \labelcref{hypSlopeLagrangian} and \Cref{slopeConditionCorollary} in \Cref{sectionSlopeOrigin}).

Since our aim is not to care much about the dependence on \(x\), we shall impose a spatial quasi-homogeneity condition:
\begin{enumerate}[\textnormal{(H}$_1$\textnormal{)},leftmargin=*,resume]
\item\label{hypQuasiHomogeneity} there exists $C < +\infty$ such that for every $x,y\in \R^N, u\in \R,\xi \in \R^N$,
\[f(y,u,\xi) \leq C (f(x,u,\xi) + u).\]
\end{enumerate}

Last of all, we need the family of functions $f_\eps:\R^N\times\R\times\R^N\to[0,+\infty]$ to converge towards \(f\) in a suitable sense, namely, we assume
\begin{enumerate}[\textnormal{(H}$_1$\textnormal{)},leftmargin=*,resume]
\item\label{hypGammaCv}
$f_\eps \uparrow f$ and $f'_{\eps,-}(\cdot,0^\pm,0) \uparrow f'_-(\cdot,0^\pm,0)$ as $\eps \to 0$.
\end{enumerate}
\noindent Notice that this assumption is empty if $f_\eps$ does not depend on $\eps$.

Our main result is the following:
\begin{theorem}\label{mainGammaConvergence}
If $(f_\eps)_{\eps > 0}$ satisfies \labelcref{hypGammaCv} with each \(f_\eps\) satisfying \labelcref{hypLowerContinuous}--\labelcref{hypCompact} and the limit \(f\) satisfying \labelcref{hypSlope}--\labelcref{hypQuasiHomogeneity}, then \(\mass^{H_f}\) is the \(\Gamma\)-limit as \(\eps\to 0\) of the functionals \(\ener_\eps\), defined in \labelcref{mainEnergyEpsilon}, for both the weak convergence and the narrow convergence of measures.
\end{theorem}
In particular, as a \(\Gamma\)-limit, the functional \(\mass^{H_f}\) must be lower semicontinuous for the weak convergence of measures (and so for the narrow convergence as well). This implies that \(H_f\) is lower semicontinuous on \(\R^N\times\R\) (see \Cref{relaxationMass}).

We point out that for the \(\Gamma-\limsup\), we need weaker assumptions on \(f_\eps\) and \(f\) (see \Cref{limsup_upper_bound}), which will be useful for some applications (see \Cref{droplets}).

We will allow ourselves slight abuses of notation. First of all, we will sometimes consider Lagrangians defined on $\R \times \R^N$ which do not depend on $x$ and still refer to hypotheses \labelcref{hypLowerContinuous}--\labelcref{hypGammaCv} ; we will use the notation $\ener_f$ instead of $\ener_f^x$ in \labelcref{def_ener} and consider $H_f$ as a function of $m$ only in \labelcref{def_Hf}. Besides, we will often identify functions $u\in L^1_{loc}$ with the measures $\mu = u\mathcal L^N$, so as to concisely write $\ener_f(u)$ instead of $\ener_f(u\lbm^N)$. Finally, we will also consider Lagrangians defined only for $u\in \R_+$, which may be thought as defined for $u\in \R$, set to $+\infty$ when $u$ is negative\footnote{Notice that if any of our assumptions is satisfied for a Lagrangian defined for $u \in \R_+$, then it holds also for the Lagrangian extended to $\R$ in this way.}. The resulting minimal cost function $H_f$ and its associated $H_f$-mass may be thought as defined on $\R_+$ and $\Mspace_+(\R^N)$ respectively, as they will be infinite on the respective complements.

\subsection{Remarks and applications}\label{sectionExamplesIntroduction}

We start with two situations where the expected $\Gamma$-convergence fails and which justify the importance of \labelcref{hypZero} and \labelcref{hypSlope}, then we provide examples and applications of our result, as a short summary of \Cref{sectionExamples}, where full details are provided. We restrict our attention to positive measures and Lagrangians defined for $u\geq 0$.

\subsubsection*{Vanishing parts do not contribute to energy.}
By assumption \labelcref{hypZero} no energy is given to any set where a function $u$ vanishes. It is a necessary condition for $\mass^{H_f}$ to be lower semicontinuous (a necessary condition to be a $\Gamma$-limit) and not identically $+\infty$. Indeed if $\mass^{H_f}$ is lower semicontinuous and finite for some measure $\mu \in \Mspace(\R^N)$ then  (see \Cref{without_f_zero_at_zero}) $\mass^{H_f}(t \mu) \leq \mass^{H_f}(\mu)$ for every $t\in (0,1)$ hence $\mass^{H_f}(0) \leq \liminf_{t\to 0^+} \mass^{H_f}(t\mu) \leq \mass^{H_f}(\mu) <+\infty$. Thus $\mass^{H_f}$ is not identically $+\infty$ if and only if $\mass^{H_f}(0) < +\infty$, i.e. $\int_{\R^N} H_f(x,0)\dd\hdm^0(x) < +\infty$. But since $H_f(x,0) = (+\infty) \times f(x,0,0)$ this can only happen if $f(\cdot,0,0) \equiv 0$. This justifies imposing \labelcref{hypZero}.

\subsubsection*{Scale invariant Lagrangians.} In \Cref{sectionCounterExamples}, we will see that in the particular case \(f_\varepsilon(x,u,\xi)=u^{-p(1-\frac{1}{p^\star})}\abs{\xi}^p\) for \(p\in (1,N)\) and \(p^\star=\frac{pN}{N-p}\), our \(\Gamma\)-convergence result does not hold as a consequence of the fact that \(\mathcal{E}_\varepsilon\) does not depend on \(\varepsilon\). Note that in this case, the slope assumption \labelcref{hypSlope} does not hold, and we also provide a simple variant of such energies which satisfies all our assumptions except \labelcref{hypSlope} where the $\Gamma$-convergence towards $\mass^{H_f}$ also fails, thus justifying the need for such a slope condition.

\subsubsection*{Concave \(H\)-masses in dimension one.} Consider the energy given for $\mu = u\lbm^N$ by
\[\ener_f(\mu) = \int_{\R^N} \abs{\nabla u}^2 + c(u)\quad\text{with Lagrangian}\quad f(x,u,\xi) = \abs{\xi}^2 + c(u).\]
In dimension $N =1$, it is shown in \cite{wirthPhaseFieldModels2019} that for any concave continuous function \(H\) with \(H(0)=0\), there exists a suitable $c \geq 0$ such that \(H_f=H\). As explained in \Cref{generalCosts1d}, \Cref{mainGammaConvergence} implies that the rescaled energies
\begin{equation}
\ener_\eps(\mu) = \int_{\R^N} f(\eps^N u,\eps^{N+1}\nabla u) \eps^{-N}= \int_{\R^N} \eps^{N+2}\abs{\nabla u}^2+\eps^{-N}c(\eps^N u).
\end{equation}
$\Gamma$-converge to $\mass^{H}$, leading to an elliptic approximation of any concave $H$-mass in dimension one. In dimension \(N\geq 2\), we will show that $H_f$ must be concave on $[0,+\infty)$, and strictly concave after the possible initial interval where it is linear (see \Cref{propositionStricConcave}) ; however, we have no solution to the inverse problem, consisting in characterizing the class of attainable minimal cost functions \(H=H_f\) for Lagrangians $f$ satisfying our assumptions.

\subsubsection*{Homogeneous \(H\)-masses in any dimension.} We consider the functional given for $\mu = u \lbm^N$ by
\begin{equation}
\ener_{f}(\mu) = \int_{\R^N} f(u,\nabla u) = \int_{\R^N} \abs{\nabla u}^p + u^s,\quad p>1,\  s\in (-p',1].
\end{equation}
Then, the rescaled energies 
\[\ener_\eps(\mu) = \int_{\R^N} f(\eps^N u,\eps^{N+1}\nabla u) \eps^{-N}= \int_{\R^N} \eps^{pN+p-N}\abs{\nabla u}^p+\eps^{-(1-s)N}u^s\]
$\Gamma$-converge to a non-trivial multiple of some $\alpha$-mass $\mass^\alpha \coloneqq \mass^{t\mapsto t^\alpha}$ where the exponent $\alpha = (1 -\frac sp + \frac sN)(1 -\frac sp + \frac 1N)^{-1}$ ranges over $(1-\frac{2}{N+1},1]$ when $(s,p)$ varies in its range and $N\geq 1$. More details are given in \Cref{homogeneousCosts}.

\subsubsection*{Space-inhomogeneous $H$-masses.} We may consider for example functionals given, for $\mu = u\lbm^N$, by
\[\ener_f(\mu) = \int_{\R^N} a(x) g(u(x),\nabla u(x)) \dd x,\]
where $a : \R^N \to \R$ is a continuous function valued in $[M^{-1}, M]$ for some $M\in (0,+\infty)$ and $g : \R \times \R^N \to [0,+\infty]$ satisfies our assumptions \labelcref{hypLowerContinuous}-\labelcref{hypSlope}, e.g. $g(u,\xi) = \abs{\xi}^p+ u^s$ with $s\in (-p',1]$) as above. In the latter case we obtain $\Gamma$-convergence towards a space-inhomogeneous $\alpha$-mass for some $\alpha \in (0,1]$ given by
\[\mass^{H_f}(\mu) = %C \sum_{1\leq i\leq \ell} a(x_i) m_i^\alpha\]
%remove for latexdiff
\begin{dcases*}
    C \sum_{1\leq i\leq \ell} a(x_i) m_i^\alpha& if $\mu = \sum_{i=1}^\ell m_i \delta_{x_i}$,\\
    +\infty & otherwise,
\end{dcases*}%endremove
\]
for some constant $C\in (0,+\infty)$.

\subsubsection*{Cahn-Hilliard approximations of droplets models.} Following the works of \cite{bouchitteTransitionsPhasesAvec1996,dubsProblemesPerturbationsSingulieres1998}, we consider the functionals
\begin{equation}
\mathcal W_\eps(u) = 
\int_{\R^N} \eps^{-\rho} (W(u) + \eps\abs{\nabla u}^2),
\end{equation}
where $W(t) \sim_{t\to+\infty} t^s$ for some exponent $s\in (-2,1)$. As shown in \Cref{droplets}, we way rewrite these functionals to fit our general framework, and recover known $\Gamma$-convergence results, under slightly more general assumptions, as stated in \Cref{GammaConvergenceW}. The $\Gamma$-limit is a non-trivial multiple of the $\alpha$-mass with $\alpha = \frac{1-s/2+s/N}{1-s/2+1/N}$.

\subsubsection*{Elliptic approximations of Branched Transport.}
The energy of Branched Transport (see \cite{bernotOptimalTransportationNetworks2009} for an account of the theory), in its Eulerian formulation, is an $H$-mass defined this time on vector measures $w$ whose divergence is also a measure,
\begin{equation}
\mass_1^H(w)\coloneqq\int_\Sigma H(x,\theta(x))\dd\hdm^1(x)+\int_{\R^d} H'(x,0^+)\dd\abs{w^\perp},
\end{equation}
where $w = \theta \xi \cdot \hdm^1 \mres \Sigma + w^\perp$ is the decomposition of $w$ into its $1$-rectifiable and $1$-diffuse parts (see \Cref{branchedTransport} for more details). An elliptic approximation of Modica-Mortola type has been introduced in \cite{oudetModicaMortolaApproximationBranched2011} for $H(m) = m^\alpha, \alpha \in (0,1)$, and their $\Gamma$-convergence result in dimension $d=2$ has been extended to any dimension in \cite{monteilEllipticApproximationsSingular2015} by a slicing method which relates the energy of $w$ to the energy of its slicings. The same slicing method, together with \Cref{mainGammaConvergence}, would allow to prove the $\Gamma$-convergence of the functionals
\begin{equation}
\ener_\eps(w)=
\begin{cases}
\int_{\R^d}f_\eps(x,\eps^{d-1} \abs{v}(x), \eps^{d}\abs{\nabla v}(x)) \eps^{1-d} \dd x &\text{if \(w = v\lbm^d , v\in W^{1,1}_\loc(\R^d,\R^d)\),}\\
+\infty&\text{otherwise,}
\end{cases}
\end{equation}
toward \(\mass_1^{H_f}\) for Lagrangians \(f_\eps\to f\) satisfying \labelcref{hypLowerContinuous}--\labelcref{hypGammaCv}, thus covering a wide range of concave $H$-masses over vector measures with divergence.

\subsection{Structure of the paper} In \Cref{sectionMinimalCost}, we prove the concavity of the minimal cost function \(H_f\) with respect to the mass variable \(m\) in full generality (\Cref{concaveDN}), we establish useful properties of general \(H\)-masses, and we identify the slope at the origin of \(H_f\) in terms of \(f\) under our assumption (\Cref{slopeUpperBound} and \Cref{slopeUpperBound1D}). In \Cref{sectionLowerBound}, we apply a concentration-compactness principle to provide a profile decomposition theorem for sequences of positive measures (\Cref{profileDecomposition}), which is used to obtain our main lower bound for the energy \(\ener_f\) (\Cref{almostMinimisers}) and also yields an existence criterion for profiles with minimal energy under a mass constraint (\Cref{existenceProfile}). \Cref{sectionGammaConvergence} is dedicated to proving lower and upper bounds on the rescaled energies \(\ener_\eps\) (\Cref{lowerBound} and \Cref{limsup_upper_bound}) that imply in particular our main \(\Gamma\)-convergence result (\Cref{mainGammaConvergence}). Last of all, in \Cref{sectionExamples}, we provide counterexamples and several examples of energy functionals that fall into our framework, as summarized in the previous section.

\section{Minimal cost function and   \texorpdfstring{\(H\)-mass}{H-mass}}\label{sectionMinimalCost}

In this section, we study the properties of general $H$-masses, of costs $H_f$ associated with general Lagrangians $f$, and we relate the slope of $H_f$ at $m=0$ to that of $f$ at $(u,\xi) = (0,0)$ in the variable $u$, under particular conditions.

\subsection{Concavity and lower semicontinuity of the minimal cost function}

Before proving \Cref{concaveDN}, let us note that it covers the case where we have a constraint \((u,\nabla u)\in A\), where \(A\subseteq\R\times \R^N\) is Borel measurable, by considering Lagrangians \(f\) taking infinite values.

\begin{proof}[Proof of \Cref{concaveDN}]
Let us first assume that $f(u,\xi) = +\infty$ when $u< 0$, so as to restrict ourselves to non-negative functions. We first prove that \(H_f\) is concave on \((0,+\infty)\).
Let \(m>0\) and \(u\in W^{1,1}_\loc\cap L^1(\R^N,\R_+)\) such that \(\int_{\R^N}u=m\). We pick a non-zero vector \(v\in\R^N\) and for every \(t\in\R\), we set \(u^t(\cdot)=u(\cdot+tv)\) and
\[
u\wedge u^t(\cdot)=\min\{u(\cdot),u^t(\cdot)\},\quad u\vee u^t(\cdot)=\max\{u(\cdot),u^t(\cdot)\}.
\]
We have \(u\wedge u^t+u\vee u^t=u+u^t\). Hence
\begin{equation}
\label{massAdditivity}
\int_{\R^N}u\wedge u^t+\int_{\R^N}u\vee u^t=2\int_{\R^N}u=2m.
\end{equation}

Moreover, it is standard that \(u\wedge u^t=u-(u^{t}-u)_-\in W^{1,1}_\loc(\R^N)\) with \(\nabla (u\wedge u^t)=\nabla u\) a.e. in \(\{u\leq u^t\}\) and \(\nabla (u\wedge u^t)=\nabla u^t\) a.e. in \(\{u>u^t\}\). Since \(u\vee u^t=u+u^t-u\wedge u^t\), we have similar identities for \(u\vee u^t\), and we obtain
\begin{equation}
\label{energyAdditivity}
\ener_f(u\wedge u^t)+\ener_f(u\vee u^t)=\ener_f(u)+\ener_f(u^t)=2\ener_f(u).
\end{equation}
Now, let \(M:t\mapsto\int_{\R^N}u\wedge u^t\). In view of \labelcref{massAdditivity}, \labelcref{energyAdditivity}, and by definition of \(H\), we have proved
\begin{equation}
H_f(M(t))+H_f(2m-M(t))\leq 2\ener_f(u).
\end{equation}
Now, by continuity of translations in \(L^1\) and since the map \((x,y)\mapsto x\wedge y\) is Lipschitz on \(\R^2\), we have that \(M\) is continuous on \(\R\) with \(M(0)=m\). Moreover \(\lim_{t\to +\infty}M(t)=0\) by dominated convergence.
So, by the intermediate value theorem \(M(\R)\supseteq (0,m]\). Hence, we have proved \(H_f(\theta)+H_f(2m-\theta)\leq 2\ener_f(u)\) for every \(\theta\in(0,m]\). Taking the infimum over \(u\) such that \(\int_{\R^N}u=m\), we obtain
\[
\frac{H_f(\theta)+H_f(2m-\theta)}{2}\leq H_f(m),\quad \forall\theta\in(0,m],
\]
that is, \(H_f\) is midpoint concave on \((0,+\infty)\). Since \(H_f\) is also bounded below (by \(0\)), we can deduce that \(H_f\) is concave \((0,+\infty)\) (see \cite[Section~72]{robertsConvexFunctions1973}), and since $H_f \geq 0$, either $H_f$ is identically $+\infty$ on $(0,+\infty)$, or it is finite everywhere, continuous, concave and non-decreasing on $(0,+\infty)$.

We now justify that if \(H_f(m)<+\infty\) for some \(m>0\) and $f(0,0) = 0$, then \(\lim_{m\to 0^+}H_f(m) = 0 = H_f(0)\).  Let \(u\in W^{1,1}_\loc(\R^N,\R_+)\) such that \(\int_{\R^N}u=m>0\) and \(\ener_f(u)<+\infty\), and set
\[
t_\ast\coloneqq\sup\{t\geq 0\st M(t)>0\}\in [0,+\infty], \quad\text{where}\quad M(t)=\int_{\R^N}u\wedge u^t.
\]
Since \(M\) is continuous with \(M(0)=\int_{\R^N}u>0\) and \(\lim_{t\to +\infty}M(t)=0\) as seen above, we have that \(t_\ast\in (0,+\infty]\), \(\lim_{t\to t_\ast}M(t)=0\) and $M(t)$ does not vanish identically near $t_\ast$. Moreover, if \(t_\ast=+\infty\), since \(u^t\to 0\) locally in measure, by dominated convergence,
\begin{align*}
\limsup_{m\to 0^+}H_f(m)\leq \limsup_{t\to +\infty}\ener_f(u\wedge u^t)
&=\limsup_{t\to +\infty}\int_{\{u<u^t\}}f(u,\nabla u)+\int_{\{u^{-t}\geq u\}}f(u,\nabla u)\\
&\leq 2f(0,0) \abs{\{u=0\}} = 0.
\end{align*}
If \(t_\ast<+\infty\), we have \(u\wedge u^{t_\ast}=0\) a.e. and \(u^t\to u^{t_\ast}\) locally in measure as \(t\to t_\ast\) by continuity of translation in \(L^1\). Thus using dominated convergence again,
\begin{align*}
\limsup_{m\to 0^+}H_f(m)\leq\limsup_{t\to (t_\ast)^-}\ener_f(u\wedge u^t)
&=\limsup_{t\to (t_\ast)^-}\int_{\{u<u^t\}}f(u,\nabla u)+\int_{\{u^{-t}\geq u\}}f(u,\nabla u)\\
&=\int_{\{u< u^{t_\ast}\}}f(u,\nabla u)+\int_{\{u^{-t_\ast}\geq u\}}f(u,\nabla u)\\
&= \ener_f(u\wedge u^{t_\ast}) = f(0,0) \times(+\infty) = 0.
\end{align*}

Finally, we treat the general case (without assuming that $f(u,\xi) = +\infty$ if $u<0$) For this, notice that for every $u\in W^{1,1}_\loc\cap L^1(\R^N)$ such that $\int_{\R^N}u=m \geq 0$, denoting $u_+$ the positive part of $u$ and $m_+ = \int_{\R^N} u_+$, since $f(0,0) = 0$ we have $\int_{\R^N} f(u,\nabla u) \geq \int_{\R^N} f(u_+,\nabla u_+) \geq H_g(m_+)$, where $g$ is defined by $g(u,\xi) = f(u,\xi)$ if $u\geq 0$ and is set to $+\infty$ if $u<0$. Since $H_g$ is non-decreasing and $m_+ \geq m$, it yields $H_f(m) \geq H_g(m_+) \geq H_g(m)$. The another inequality $H_f(m) \leq H_g(m)$ being trivial, $H_f = H_g$ on $\R_+$.
\end{proof}
\begin{remark}\label{without_f_zero_at_zero}
    Let us remark two things about the proof.
    
    First, we actually proved the concavity and monotonicity of
    \begin{equation*}
m\mapsto H_f(m) \coloneqq \inf\left\{\int_{\R^N}f(u,\nabla u)\st u\in W^{1,1}_\loc(\R^N,\R_+),\,\int_{\R^N}u=m\right\}
\end{equation*}
on  $(0,+\infty)$ for $f$ Borel, without assuming $f(0,0) = 0$.

Second, the end of the proof shows that under this extra assumption, minimizing among signed profiles or non-negative (resp. non-positive) profiles is equivalent when $m \geq 0$ (resp $m \leq 0)$.
\end{remark}

\subsection{Definition and relaxation of the \texorpdfstring{$H$-mass}{H-mass}}\label{sectionHmass}

\begin{definition}\label{defHmass}
Let \(H:\R^N\times\R\to[0,+\infty]\) be a Borel measurable function having left/right slopes at the origin defined for each $x\in\R^N$ by
\begin{equation}
\label{slopeDefinition}
H'(x,0^\pm)\coloneqq\lim_{m\to 0^\pm}\frac{H(x,m)}{\abs{m}} \in [0,+\infty].
\end{equation}

For every finite signed Borel measure \(\mu\in \Mspace(\R^N)\), we set
\[
H(\mu)\coloneqq H(\cdot,\mu(\{\cdot\}))\hdm^0+H'(\cdot,0^+) \mu^d_++H'(\cdot,0^-) u^d_-,
\]
where \(\mu = \mu^a+\mu^d\) is the decomposition of $u$ into its atomic part $\mu^a$ and its diffuse (or non-atomic) part $\mu^d$, $\mu^d=\mu^d_+-\mu^d_-$ is the Jordan decomposition of $\mu^d$ into positive and negative parts.

The \emph{$H$-mass} of $\mu$ is then defined as the total variation of $H(\mu)$, that is:
\begin{multline*}
\mass^H(\mu)\coloneqq \int_{\R^N} H(x,\mu(\{x\})) \dd\hdm^0(x)\\
+ \int_{\R^N} H'(x,0^+) \dd \mu^d_+(x) + \int_{\R^N} H'(x,0^-) \dd \mu^d_-(x).
\end{multline*}
\end{definition}
The previous definitions and notations extend in the obvious way to the case of functions \(H:\R\to[0,+\infty]\) with no space variable \(x\), interpreted as functions independent from \(x\).

\(\mass^H\) is a natural spatially non-homogeneous extension (depending on the position $x$) of the $H$-mass of $k$-dimensional flat currents\footnote{In the case $k=0$, since signed measures are merely $0$-currents with finite mass.} from Geometric Measure Theory, introduced by \cite{flemingFlatChainsFinite1966} (see also the more recent works \cite{depauwSizeMinimizationApproximating2003,colomboLowerSemicontinuousEnvelope2017}). 

We say that $H : \R^N \times \R \to [0,+\infty]$ is \emph{mass-concave} if $m\mapsto H(x,m)$ is concave on $(0,+\infty)$ and $(-\infty,0)$ for every $x\in\R^N$. From \cite{bouchitteNewLowerSemicontinuity1990}, we have the following result\footnote{In the notation of this paper, we take \(\mu=0\) and \(f(x,s)=\abs{s}^2\); we have \(\varphi_{f}(x,0)=0\) and \(\varphi_{f}(x,s)=+\infty\) if \(s\neq 0\).}:
\begin{proposition}[{\cite[Theorem~3.3]{bouchitteNewLowerSemicontinuity1990}}]\label{lscHmass}
Assume that \(H:\R^N\times\R\to[0,+\infty]\) is lower semicontinuous, mass-concave and $H(\cdot,0) \equiv 0$. Then \(\mass^H\) is sequentially lower semicontinuous on \(\Mspace(\R^N)\) for the weak topology.
\end{proposition}
From another work from the same authors \cite[Theorem~3.2]{bouchitteRelaxationClassNonconvex1993}, we know that under some further assumptions on $H$, \(\mass^H\) is the relaxation for the weak topology of the functional
\[
\mass_\mathrm{atom}^H(\mu)=
\begin{cases}
\sum_{i=1}^k H(x_i,m_i)&\text{if \(\mu=\sum_{i=1}^k m_i\delta_{x_i}\) with $k\in \N^\ast$, $x_i\in\R^N$, \(m_i \in\R\),}\\
+\infty&\text{otherwise.}
\end{cases}
\]
We need a slightly different result\footnote{In \cite[Theorem~3.2]{bouchitteRelaxationClassNonconvex1993}, $H$ is assumed to be lower semicontinuous and the authors make a further coercivity assumption (assumption (3.5) in the paper) that we want to avoid.}, namely that for any function \(H:\R^N\times\R\to[0,+\infty]\) satisfying all the assumptions of \Cref{lscHmass} except the lower semicontinuity, the relaxation of \(\mass_\mathrm{atom}^H\) for the \emph{narrow} sequential convergence is \(\mass^{H_\lsc}\), where \(H_\lsc\) is the lower semicontinuous envelope of \(H\), which can be expressed as
\begin{align}\label{semicontinuousEnvelope}
H_\lsc(x,m)&=
\sup\{G(x,m)\st G\leq H\text{ with \(G\) lower semicontinuous}\}.
\end{align}
It is worth noticing that if $H(\cdot,0) \equiv 0$ and $H$ is mass-concave, then these properties hold also for $H_\lsc$.

\begin{proposition}\label{relaxationMass}
Let \(H:\R^N\times\R\to[0,+\infty]\) be a function which is mass-concave and such that \(H(\cdot,0)\equiv 0\). Then, the sequentially lower semicontinuous envelope of \(\mass_\mathrm{atom}^H\) in the narrow topology of \(\Mspace(\R^N)\) is given by \(\mass^{H_\lsc}\), namely we have:
\begin{equation}
\label{relaxationFormula}
\mass^{H_\lsc}=\sup\Bigl\{F\st F\leq \mass_\mathrm{atom}^H\text{, \(F\) sequentially narrowly l.s.c. on \(\Mspace(\R^N)\)}\Bigr\}.
\end{equation}
\end{proposition}

We point out that for a general $H$, for \(\mass^H\) to be sequentially lower semicontinuous (for the narrow topology) it is necessary that $H$ is lsc on $\R^N \times (0,+\infty)$. However, neither the subadditivity of \(H\) in $m$ nor its lower semicontinuity on $\R^N \times \R_+$ are necessary. Indeed, \(\mass^H\) is sequentially lower semicontinuous if for instance \(H(x,m)=+\infty\) when \(x\neq 0, m > 0\), \(H(x,0) = 0\) when $x\neq 0$ and \(H(0,\cdot)\) is any lower semicontinuous function. Nevertheless the subadditivity in the mass $m$ and the lower semicontinuity would be necessary if \(H\) did not depend on \(x\).

\begin{proof}[Proof of \Cref{relaxationMass}]
Since \(H_\lsc\) is lower semicontinuous and mass-concave, we know from \Cref{lscHmass} that \(\mass^{H_\lsc}\) is sequentially lower semicontinuous in the weak topology hence also in the narrow topology of \(\Mspace(\R^N)\). Since \(\mass^{H_\lsc}\leq\mass_\mathrm{atom}^H\), we deduce that \(\mass^{H_\lsc}\) is lower or equal than the sequentially lower semicontinuous envelope of \(\mass_\mathrm{atom}^H\) in the narrow topology, i.e. the right hand side in \labelcref{relaxationFormula}, which we denote by \(F:\Mspace(\R^N)\to\R_+\). We shall see that \(F\leq\mass^{H_\lsc}\).

We first prove that
\begin{equation}\label{FleqHlsc}
F\leq\mass_\mathrm{atom}^{H_\lsc}.
\end{equation}
For this, we let \(\mu=\sum_{i=1}^km_i\delta_{x_i}\) be a finitely atomic positive measure and we let \(\mu_n\coloneqq\sum_{i=1}^k m_{i,n}\delta_{x_{i,n}}\) where for each \(i\in\{1,\dots,k\}\), \((x_{i,n})_{n\in\N}\) is a sequence of points converging to \(x_i\) and \((m_{i,n})_{n\in\N}\) is a sequence converging to \(m_i\) such that \({H_\lsc}(x_i,m_i)=\lim_{n\to\infty}H(x_{i,n},m_{i,n})\). Then \((\mu_n)_{n\in\N}\) converges narrowly to \(\mu\) and, by lower semicontinuity,
\begin{align*}
F(\mu)\leq\liminf_{n\to\infty}F(\mu_n)&\leq\liminf_{n\to\infty}\mass_\mathrm{atom}^H(\mu_n)\\
&=\lim_{n\to\infty}\sum_{i=1}^k H(x_{i,n},m_{i,n})=\sum_{i=1}^k{H_\lsc}(x_i,m_i),
\end{align*}
so that \(F(\mu)\leq\mass_\mathrm{atom}^{H_\lsc}(\mu)\) as wanted.

We now prove that \(F\leq\mass^{H_\lsc}\). Let \(\mu\in\Mspace(\R^N)\) and \(\mu = \mu^a+\mu^d\) be the decomposition of $\mu$ into its atomic part $\mu^a = \sum_{i=1}^k m_i\delta_{x_i}$, with $k\in\N\cup\{+\infty\}$ (here, \(k=0\) if there is no atom), and its diffuse part $\mu^d$, and let $\mu^d=\mu^d_+-\mu^d_-$ be the Jordan decomposition of $\mu^d$ into positive and negative parts. We then discretize \(\mu^d_\pm\) by taking \(n\in\N^\ast\), a partition \((Q^n_i)_{i\in\{1,\dots,(n2^n)^N\}}\) of \([-n,n)^N\) by means of cubes of the form \(Q^n_i=c^n_i+2^{-n}[-1,1)^N\) with \(c^n_i\in\R^N\), and we define
\[
\mu_n\coloneqq\sum_{i=1}^{n\wedge k} m_i\delta_{x_i}+\sum_{i=1}^{(n2^n)^N}\mu^d_+(Q^n_i)\delta_{x^n_i}-\sum_{i=1}^{(n2^n)^N}\mu^d_-(Q^n_i)\delta_{y^n_i},
\]
where for each \(i\in\{1,\dots,(n2^n)^N\}\), \(x^n_i,y^n_i\in \Bar Q^n_i\) are some points such that
\begin{equation}
\label{riemannIntegral}
{H_\lsc}'(x^n_i,0^+)=\inf_{x\in \Bar Q^n_i} {H_\lsc}'(x,0^+),\quad
{H_\lsc}'(y^n_i,0^-)=\inf_{x\in \Bar Q^n_i} {H_\lsc}'(x,0^-).
\end{equation}
Such points exist since \(\Bar Q^n_i\) is compact and since by concavity,
\begin{equation}\label{supSlope}
{H_\lsc}'(x,0^\pm) = \sup_{\pm m > 0} \frac{H_\lsc(x,m)}{\abs{m}},
\end{equation}
so that $H_\lsc'(\cdot,0^\pm)$ are lower semicontinuous as suprema of lower semicontinuous functions.

The sequence \((\mu_n)_{n\in\N}\) converges narrowly to \(\mu\). We deduce by lower semicontinuity of $F$,
\begin{align*}
F(\mu)&\leq \liminf_{n\to\infty} F(\mu_n)\overset{\labelcref{FleqHlsc}}\leq \liminf_{n\to\infty}\mass_\mathrm{atom}^\lsc(\mu_n) \\
&=\liminf_{n\to\infty}\sum_{i=1}^{n\wedge k} {H_\lsc}(x_i,m_i)+\sum_{i=1}^{(n2^n)^N}\Bigl({H_\lsc}(x^n_i,\mu^d_+(Q^n_i)) + {H_\lsc}(y^n_i,-\mu^d_-(Q^n_i))\Bigr),\\
\overset{\labelcref{supSlope}}&\leq\liminf_{n\to\infty}\begin{multlined}[t][0.75\displaywidth]\sum_{i=1}^{n\wedge k} {H_\lsc}(x_i,m_i) +\sum_{i=1}^{(n2^n)^N}{H_\lsc}'(x^n_i,0^+)\mu^d_+(Q^n_i)\\
+\sum_{i=1}^{(n2^n)^N}{H_\lsc}'(y^n_i,0^-)\mu^d_-(Q^n_i)\end{multlined}\\
\overset{\labelcref{riemannIntegral}}&\leq\liminf_{n\to\infty}\sum_{i=1}^{n\wedge k} {H_\lsc}(x_i,m_i)
+\sum_{i=1}^{(n2^n)^N}\int_{Q^n_i}{H_\lsc}'(\cdot,0^+)\dd \mu^d_++\int_{Q^n_i}{H_\lsc}'(\cdot,0^-)\dd \mu^d_-\\
&=\sum_{i=1}^{k} {H_\lsc}(x_i,m_i)
+\int_{\R^N}{H_\lsc}'(\cdot,0^+)\dd \mu^d_++\int_{\R^N}{H_\lsc}'(\cdot,0^-)\dd \mu^d_-\\
&=\mass^{H_\lsc}(\mu),
\end{align*}
where we have used monotone convergence in the last but one equality.
\end{proof}

\subsection{Slope at the origin of the minimal cost function}\label{sectionSlopeOrigin}

In this section we provide a technical assumption, that is simple in dimension $N=1$, on the Lagrangian $f$ and which implies the comparison condition on the slopes \labelcref{hypSlope}.

\begin{enumerate}[\textnormal{(S)},leftmargin=*]
\item\label{hypSlopeLagrangian} for every $x_0\in \R^N$,
\begin{equation}\label{slopeInf}
f'_-(x_0,0^\pm,0) = \displaystyle\liminf_{(x,u,\xi) \to (x_0,0^\pm,0)} \frac{f(x,u,\xi)}{\abs{u}} \geq 
\limsup_{u\to 0^\pm}\sup_{\abs{\xi}= 1}\frac{f(x_0,u,\rho(\abs{u})\xi)}{\abs{u}},
\end{equation}
with \(\rho\equiv 0\) if \(N=1\) and for some \(\rho\in\Cspace((0,1],(0,+\infty))\) satisfying
\[
\int_0^1\Big(\int_y^1\frac{\dd t}{\rho(t)}\Big)^N\dd y<+\infty\quad\text{if \(N\geq 2\)}.
\]
\end{enumerate}

\begin{proposition}\label{slopeUpperBound}
Let $f:\R_+\times\R^N\to [0,+\infty]$ be a lower semicontinuous function such that $f(0,0)=0$, with $N\geq 2$. For every function \(\rho\in\Cspace((0,1],(0,+\infty))\) such that
\begin{equation}
\label{hypRegularityOriginSlope}
\int_0^1\Big(\int_y^1\frac{\dd t}{\rho(t)}\Big)^N\dd y<+\infty,
\end{equation}
the function $H_f$ defined in \labelcref{PropertiesDefinitionH} satisfies
\begin{equation}
\label{slopeComparison}
\lim_{m\to 0^+}\frac{H_f(m)}{m}\leq\limsup_{u\to 0^+}\sup_{\xi\in\Sph^{N-1} }\frac{f(u,\rho(u)\xi)}{u}.
\end{equation}
\end{proposition}
\begin{proof}
For every $y\geq 0$, we let
\[
F(y)=\int_y^1\frac{\dd t}{\rho(t)} \in [0,+\infty].
\]
The function $F$ is decreasing, and belongs to $\Cspace^1((0,1])$ and \(L^N((0,1])\) by assumption. We now consider the solution of the ODE \(v_\eps'=-\rho(v_\eps)\), with \(v_\eps(0)=\eps\), given by
\[
v_\eps(r)=
\begin{cases}
F^{-1}(F(\eps)+r),&\text{if \(0\leq r<F(0)-F(\eps)\),}\\
0&\text{if \(r\geq F(0)-F(\eps)\)},
\end{cases}
\]
with $F(0)$ possibly equal to $+\infty$. Notice that $v_\eps \in W^{1,1}_\loc(\R_+)$ because it is non-increasing and bounded, hence it has finite total variation, and because it is of class $\Cspace^1$ except possibly at $r_\eps \coloneqq F(0)-F(\eps)$, where it has no jump. As a consequence the radial profile defined by \(u_\eps(x)\coloneqq v_\eps(\abs{x})\) belongs to $W^{1,1}_\loc(\R^N)$ and we compute, using the change of variables \(s=v_\eps(r)\) (i.e. \(r=F(s)-F(\eps)\)) and an integration by parts combined with monotone convergence.
\begin{align*}
m_\eps\coloneqq\int_{\R^N}u_\eps&=\abs{\Sph^{N-1}}\int_0^\infty v_\eps(r)r^{N-1}\dd r\\
&=-\abs{\Sph^{N-1}}\int_0^\eps s(F(s)-F(\eps))^{N-1}F'(s)\dd s\\
&=\abs{\Sph^{N-1}}\lim_{t\downarrow 0} \left(\int_t^\eps\frac{(F(s)-F(\eps))^N}{N}\dd s- \left[s\frac{(F(s)-F(\eps))^N}N\right]_t^\eps\right)\\
&=\abs{\Sph^{N-1}}\int_0^\eps\frac{(F(s)-F(\eps))^N}{N}\dd s\xto{\eps \to 0}0.
\end{align*}
The equality on the last line holds because $\lim_{t\to 0^+} \int_t^\eps (F-F(\eps))^N < +\infty$ (since $F \in L^N((0,1])$), hence $\lim_{t\to 0} t(F(t)-F(\eps))^N$ exists by existence of the limit in the previous line, and it must be zero (again, because $F \in L^N((0,1])$).

Moreover, since \(\sup_{[0,+\infty)} v_\eps=\eps\),
\begin{align*}
\ener(u_\eps)
=\int_0^\infty\int_{\Sph^{N-1}}f(v_\eps(r),v_\eps'(r)\xi) r^{N-1}\dd \hdm^{N-1}(\xi) \dd r
\leq 
m_\eps\sup_{u\leq\eps,\, \abs{\xi}=1}\frac{f(u,\rho(u)\xi)}{u}.
\end{align*}
By assumption, we deduce that 
\[
\limsup_{m\to 0^+}\frac{H_f(m)}{m}\leq\limsup_{\eps\to 0^+} \frac{\ener(u_\eps)}{m_\eps}\leq\limsup_{u\to 0^+}\sup_{\xi\in\Sph^{N-1} }\frac{f(u,\rho(u)\xi)}{u}.\qedhere
\]
\end{proof}

In dimension \(N=1\), we need no other assumption than \(H_f<+\infty\), as stated below.
\begin{proposition}\label{slopeUpperBound1D}
Let $f :\R_+\times\R^N\to[0,+\infty]$ be Borel measurable with $N=1$. The minimal cost function \(H_f\) is either identically infinite on \((0,+\infty)\), or it satisfies \labelcref{slopeComparison} with $\rho \equiv 0$, i.e.
\[\lim_{m\to 0^+} \frac{H_f(m)}m \leq \limsup_{u\to 0^+} \frac{f(u,0)}u.\]
\end{proposition}

\begin{proof}
One can assume that there exists \(u\in W^{1,1}_\loc(\R,\R_+)\) with \(0<\int_\R u<+\infty\) and \(\ener(u)<+\infty\). In particular, up to changing the value of \(u\) on a negligible set, \(u\) is continuous on \(\R\). Let \(\eps\in(0,\sup_\R u)\), set $A_\eps \coloneqq \{x : u(x) = \eps\}$ which is non-empty by the intermediate value theorem and integrability of $u$, and define
\begin{align*}
a_\eps &= \begin{dcases*}
\inf A_\eps & if $\inf A_\eps >-\infty$,\\
\text{any point in }  (-\infty,-\eps^{-1}) \cap A_\eps & otherwise,
\end{dcases*}\\
b_\eps &=\begin{dcases*}
\sup A_\eps & if $\sup A_\eps <+\infty$,\\
\text{any point in } (\eps^{-1},+\infty) \cap A_\eps & otherwise.
\end{dcases*}
\end{align*}
By continuity and integrability of $u$, \(u(a_\eps) = u(b_\eps) =\eps\) and
\[\sup_{x\in\R\setminus [a_\eps,b_\eps]} u(x) \leq \eps \vee \sup_{\abs{x}> \eps^{-1}} u(x) \xto{\eps \to 0} 0.\]
Moreover $a_\eps,b_\eps$ converge to points $-\infty\leq a \leq b \leq +\infty$, hence $u = 0$ on $\R\setminus(a,b)$ and by dominated convergence, since \(\nabla u=0\) a.e. on \(\{u=0\}\),
\[
+\infty > \lim_{\eps\to 0^+}\int_{\R\setminus [a_\eps,b_\eps]}u+f(u,\nabla u) = f(0,0)\lbm(\R\setminus(a,b)).
\]
Notice that this limit is necessary zero. Let \(m>0\). If \(\eps\) is small enough, then \(\int_{\R\setminus [a_\eps,b_\eps]} u<m\) so that we can take \(R_\eps>0\) such that \(\eps R_\eps=m- \int_{\R\setminus [a_\eps,b_\eps]} u\). We then define
\[
u_\eps(x)=
\begin{cases}
u(x)&\text{if \(x\leq a_\eps\),}\\
\eps&\text{if \(a_\eps< x< a_\eps+R_\eps\)},\\
u(b_\eps+x-(a_\eps+R_\eps))&\text{if \(x\geq a_\eps+R_\eps\),}
\end{cases}
\]
so that \(\int_\R v_\eps=m\). Moreover,
\[
\ener(v_\eps)=\ener(u,\R\setminus [a_\eps,b_\eps])+R_\eps f(\eps,0).
\]
Hence, as \(R_\eps=\frac{m+o(1)}{\eps}\) as \(\eps\to 0\),
\[
H_f(m)\leq \limsup_{\eps\to 0^+}\ener(v_\eps)
= m\limsup_{\eps\to 0^+}\frac{f(\eps,0)}{\eps}.\qedhere
\]
\end{proof}

From \Cref{slopeUpperBound} and \Cref{slopeUpperBound1D} we obtain the following corollary.

\begin{corollary}\label{slopeConditionCorollary}
Let $f : \R \times \R^N \to [0,+\infty]$ satisfy \labelcref{hypLowerContinuous}, \labelcref{hypZero} and \labelcref{hypSlopeLagrangian}. Assume in addition that $H_f<+\infty$ if $N=1$. Then \labelcref{hypSlope} holds.
\end{corollary}
\begin{proof}
If $N=1$ we apply \Cref{slopeUpperBound1D} to get
\[\lim_{m\to 0^+} \frac{H_f(m)} m \leq \limsup_{u\to 0^+} \frac{f(u,0)}u \overset{\labelcref{hypSlopeLagrangian}}{\leq} \liminf_{m\to 0^+} \frac{f(u,0)}u = f'_-(0^+,0),\]
and if $N=2$ taking a function $\rho$ as in \labelcref{hypSlopeLagrangian} and applying \Cref{slopeUpperBound} yields
\[\lim_{m\to 0^+} \frac{H_f(m)} m \leq \limsup_{u\to 0^\pm}\sup_{\abs{\xi}= 1}\frac{f(x_0,u,\rho(\abs{u})\xi)}{\abs{u}} \overset{\labelcref{hypSlopeLagrangian}}{\leq} \liminf_{m\to 0^+} \frac{f(u,0)}u = f'_-(0^+,0).\]
The analogous inequality when $m\to 0^-$ is obtained by considering the symmetric Lagrangian $(u,\xi) \mapsto f(-u,-\xi)$.
\end{proof}

\subsection{Strict concavity of the minimal cost function in dimension \texorpdfstring{\(N\geq 2\)}{N≥2}}

We show that in dimension \(N\geq 2\), the minimal cost function must be strictly concave away from the possible initial interval where it is linear:
\begin{proposition}\label{propositionStricConcave} 
Assume that \(N\geq 2\) and that $f :\R\times\R^N\ni (u,\xi) \mapsto f(u,\xi)\in [0,+\infty]$ satisfies \labelcref{hypLowerContinuous}, \labelcref{hypConvex}, \labelcref{hypZero}, \labelcref{hypCompact} and \labelcref{hypSlope}. Let
\[
m_\ast=\sup\{m\geq 0\st \text{\(H_f\) is linear on \([0,m]\)}\},
\]
where \(H_f\) is defined in \labelcref{PropertiesDefinitionH}. Then, \(H_f\) is strictly concave on \((m_\ast,+\infty)\). A similar statement holds on $\R_-$.
\end{proposition}
A similar result does not hold in dimension \(1\) since any continuous concave function \(H:\R_+\to\R_+\) with \(H(0)=0\) can be written as \(H=H_f\) with \(f\) satisfying all our assumptions \labelcref{hypLowerContinuous}--\labelcref{hypGammaCv} (see \Cref{generalCosts1d}).

We denote by \(\mathcal{M}^f_m\) the set of non-negative minimizers of mass \(m\in\R_+\):
\begin{equation}
\mathcal{M}^f_m\coloneqq \Bigl\{u\in W^{1,1}_\loc(\R^N,\R_+)\st \ener_f(u)=H_f(m)\text{ and }\int_{\R^N}u=m\Bigr\}.
\end{equation}
The proof of \Cref{propositionStricConcave} is based on the following observation: 
\begin{lemma}\label{lemmaOptimalitySupInf}
Let \(f:\R_+\times\R^N\to[0,+\infty]\) be Borel measurable and let \(u_i\in\mathcal{M}^f_{m_i}\) with \(m_i\in\R_+\) for \(i=1,2\). Let also \(u_\ast\coloneqq \min\{u_1,u_2\}\), \(u^\ast\coloneqq \max\{u_1,u_2\}\), \(m_\ast\coloneqq \int_{\R^N}u_\ast\) and \(m^\ast\coloneqq \int_{\R^N}u^\ast\). If \(H_f\) is affine on \([m_\ast,m^\ast]\) then \(u_\ast\in\mathcal{M}^f_{m_\ast}\) and \(u^\ast\in\mathcal{M}^f_{m^\ast}\).
\end{lemma}
\begin{proof}[Proof of \Cref{lemmaOptimalitySupInf}]
We use the same observations as in the proof of \Cref{concaveDN}. In particular, we have \(m_\ast+m^\ast=m_1+m_2\); since \(H_f\) is affine on \([m_\ast,m^\ast]\), it yields
\[
H_f(m_\ast)+H_f(m^\ast)=H_f(m_1)+H_f(m_2).
\]
But we have also
\[
H_f(m_\ast)+H_f(m^\ast)\leq\ener_f(u_\ast)+\ener_f(u^\ast)=\ener_f(u_1)+\ener_f(u_2)=H_f(m_1)+H_f(m_2),
\]
so that the inequalities we used, i.e. \(H_f(m_\ast)\leq\ener_f(u_\ast)\) and \(H_f(m^\ast)\leq\ener_f(u^\ast)\), are actually equalities.
\end{proof}
We also use an elementary Sobolev type inequality:
\begin{lemma}\label{lemmaSobolevInfinityLp}
Let \(N\geq 2\), \(p\in (1,+\infty)\) and \(\omega\subset\R^{N-1}\) be a bounded open set. For every \(u\in W_\loc^{1,p}(\R\times\omega)\),
\[
\int_\omega\norm{u(\cdot,x')}_{L^\infty(\R)}\dd x'\leq\norm{u}_{L^1(\R\times\omega)}+\abs{\omega}^{\frac{p-1}{p}}
\norm*{\frac{\partial u}{\partial x_1}}_{L^p(\R\times\omega)}.
\]
\end{lemma}
\begin{proof}[Proof of \Cref{lemmaSobolevInfinityLp}]
We prove the lemma when \(u\in\Cspace^1(\R\times\omega)\); the general case follows by approximation.  For every \(x_1,y_1\in\R\), \(x'\in\omega\), we have
\[
u(x_1,x')=u(y_1,x')+\int_{y_1}^{x_1}\frac{\partial u}{\partial x_1}(t,x')\dd t.
\]
By averaging in the variable \(y_1\), we deduce
\[
\abs{u(x_1,x')}\leq \int_{x_1-\frac 12}^{x_1+\frac 12}\abs{u(y_1,x')}\dd y_1+\int_{x_1-\frac 12}^{x_1+\frac 12}\abs*{\frac{\partial u}{\partial x_1}(t,x')}\dd t.
\]
The result follows from H\"older inequality after integrating over \(\omega\).
\end{proof}

\begin{proof}[Proof of \Cref{propositionStricConcave}]
Assume by contradiction that the concave function \(H_f\) is not strictly concave on \((m_\ast,+\infty)\) which means that there exists \(m\in (m_\ast,+\infty)\) and \(\eta>0\) such that \(H_f\) is affine on \([m-\eta,m+\eta]\). (Note that \(\eta\leq m-m_\ast\) by definition of \(m_\ast\).) Moreover, we will see in \Cref{existenceProfile} that \(\mathcal{M}^f_m\) is not empty because $H_f$ is not linear on $[0,m]$. We let \(u\in \mathcal{M}^f_m\). 

As before, we shall use the notations \(\wedge\) and \(\vee\) for the minimum and maximum; we also let \((e_1,\dots,e_N)\) be the canonical basis of \(\R^N\). Knowing that \(\tau \mapsto u(\cdot+\tau)\) is continuous in $\R^N$ for every $u\in L^1(\R^N)$, that \(u\mapsto u(\cdot+\tau)\) is isometric in \(L^1(\R^N)\) for every \(\tau\in\R^N\), that the map \((x,y)\mapsto x\wedge y\) is Lipschitz on \(\R^N\times\R^N\), and since the set \(\mathcal{M}_{m+\frac{\eta}{2}}\) is compact in \(L^1\) up to translations in view of \Cref{compact_modulo}, we deduce that there exists \(\delta_0>0\) such that
\begin{equation}\label{UnifContTrans}
%\begin{cases}
%\norm{u\vee u(\cdot+\delta e_1)}_{L^1(\R^N)}<m+\eta&\text{for all }u\in\mathcal{M}^f_m,\\
\norm{u\wedge u(\cdot+\delta e_2)}_{L^1(\R^N)}>m\quad\text{for all }\delta\in (0,\delta_0)\text{ and }u\in\mathcal{M}^f_{m+\frac{\eta}{2}}.
%\end{cases}
\end{equation}
We now construct by induction a sequence \((t_n)_{n\in\N}\) in \(\R_+\) and a sequence \((u_n)_{n\in\N}\) in \(\mathcal{M}^f_m\) such that
\begin{equation}\label{inductionHyp}
t_{n+1}\geq t_n+\delta_0\quad \text{and}\quad u_n(x)\leq U(x)\wedge U(x+ t_n e_2)\quad\forall x\in\R^N,
\end{equation}
where we have set
\[
U(x)\coloneqq \esssup_{t\in\R} u(x+t e_1).
\]
To this aim, we first set \(u_0\coloneqq u\) and \(t_0=0\). Then, if we assume that \(t_n\) and \(u_n\) are constructed as before, we first pick an \(\delta^1_n\in \R_+\) such that 
\(
 v_n\coloneqq u_n\vee u_n(\cdot+\delta^1_n e_1)\quad\text{satisfies}\quad\int_{\R^N} v_n=m+\frac{\eta}{2},
\)
which is possible since $\eta\leq m$, as we argued in the proof of \Cref{concaveDN}. Similarly, we pick a \(\delta^2_n\in \R_+\) such that \(u_{n+1}\coloneqq v_n\wedge v_n(\cdot+\delta^2_n e_2)\) satisfies \(\int_{\R^N} u_{n+1}=m\), and we set \(t_{n+1}=t_n+\delta^2_n\). By \Cref{lemmaOptimalitySupInf}, $v_n\in\mathcal{M}^f_{m+\frac{\eta}{2}}$ and $u_n\in\mathcal{M}^f_m$. By \labelcref{UnifContTrans}, we have \(\delta^2_n\geq \delta_0\), thus insuring the first condition in \labelcref{inductionHyp}. For the second condition, we observe that for all \(x=(x_1,x')\in\R^N\),
\begin{align*}
u_{n+1}(x)\leq(\sup_{t\in\R} u_n(x+t e_1))\wedge (\sup_{t\in\R}u_n(x+t e_1+\delta^2_n e_2))\leq U(x)\wedge U(x+t_{n+1}e_2),
\end{align*}
where in the last inequality we have used the induction hypothesis \labelcref{inductionHyp}.

We now show that the sequence \((u_n\lbm^N)_{n\in\N}\) is vanishing which will contradict the compactness of \(\mathcal{M}^f_m\) in \(L^1\) up to translations.

For this, we let \((x_k)_{k\in\N}\) be a sequence in \(\R^N\) and \((u_{n_k})_{k\in\N}\) be a subsequence of \((u_n)_{n\in\N}\) such that
\[
\limsup_{n\to\infty}\sup_{x\in\R^N}\int_{x+[0,1)^N}u_n=\lim_{k\to\infty}\int_{x_k+[0,1)^N}u_{n_k}.
\]

By \labelcref{hypCompact}, we have \(\frac{\partial u}{\partial x_1}\in L^p(\R^N)\). Using this fact, estimate \labelcref{inductionHyp}, and \Cref{lemmaSobolevInfinityLp} with $\omega$ a unit cube in $\R^{N-1}$, we obtain
\begin{multline*}
\lim_{k\to\infty}\int_{x_k+[0,1)^N}u_{n_k}\leq \liminf_{k\to\infty}\int_{x_k+[0,1)^N}U\wedge\int_{x_k+t_{n_k} e_2+[0,1)^N} U\\
\leq \liminf_{k\to\infty} \biggl(\norm{u}_{L^1(\{0\leq (x-x_k)\cdot e_2\leq 1\})}+
\norm*{\frac{\partial u}{\partial x_1}}_{L^p(\{0\leq (x-x_k)\cdot e_2\leq 1\})}\biggr)\\
\wedge\liminf_{k\to\infty} \biggl(
 \norm{u}_{L^1(\{t_{n_k}\leq (x-x_k)\cdot e_2\leq t_{n_k}+1\})}+
\norm*{\frac{\partial u}{\partial x_1}}_{L^p(\{t_{n_k}\leq (x-x_k)\cdot e_2\leq t_{n_k}+1\})} \biggr),
\end{multline*}
and the conclusion follows since the sequences \((x_k\cdot e_2)_{k\in\N}\) and \((t_{n_k}+x_k\cdot e_2))_{k\in\N}\) cannot be both bounded as \(\lim_{k\to\infty}t_{n_k}=\infty\).
\end{proof}

\section{Lower bound for the energy and existence of optimal profiles}\label{sectionLowerBound}

Our main tool to localize the energy and obtain a lower bound relies on a profile decomposition for bounded sequences of positive measures, which is reminiscent of the concentration-compactness principle of P.-L. Lions. This differs from classical strategies to localize the energy which are based on suitable cut-offs. Naturally, this concentration-compactness result also provides a criterion for the existence of optimal profiles in \labelcref{def_Hf}. Nothing can be said beyond existence of a minimizer at this level of generality. Further properties such as uniqueness and radial symmetry would require conditions on the Lagrangian $f$ and not only on the cost function $H_f$. We deal with these questions in a particular case in \Cref{homogeneousCosts}.

\subsection{Profile decomposition by concentration-compactness}

We prove a profile decomposition theorem for bounded sequences of positive measures over $\R^N$, which is essentially equivalent to \cite[Theorem~1.5]{marisProfileDecompositionSequences2014} in the Euclidean case. We have added an extra information on mass conservation that will be useful, and provide a self-contained simple proof. We start with a definition.

\begin{definition}\label{DefinitionVanishing}
A sequence of positive measures $(\mu_n)_{n\in\N} \in \Mspace_+(\R^N)$ is \emph{vanishing} if
\[\sup_{x\in \R^N} \mu_n(B_1(x)) \xto{n\to \infty} 0.\]
\end{definition}

Any bounded sequence of positive measures over $\R^N$ may be decomposed (up to subsequence) into a countable collection of narrowly converging \enquote{bubbles} and a vanishing part, accounting for the total mass of the sequence, as stated in the following theorem.

\begin{theorem}\label{profileDecomposition}
For every bounded sequence \((\mu_n)_{n\in\N}\) of positive Borel measures on \(\R^N\), there exists a subsequence \((\mu_n)_{n \in \sigma(\N)}\), $\sigma \in \Sigma$, a non-decreasing sequence of integers $(k_n)_{n \in \sigma(\N)}$ converging to some $k \in \N \cup \{+\infty\}$, a sequence of non-trivial positive Borel measures $(\mu^i)_{0\le i<k}$, and for every $n \in  \sigma(\N)$, a collection of balls $(B^i_n)_{0 \leq i < k_n}$ centered at points of $\supp \mu_n$ such that, writing for all $n\in  \sigma(\N)$,
\begin{equation}
\mu_n = \mu^b_n + \mu^v_n,\quad \text{where } \mu^b_n = \sum_{0\le i < k_n} \mu_n \mres B_n^i,
\end{equation}
\begin{enumerate}[(A)]
\item\label{bubbleEmergence} bubbles emerge: $(c_{B_{n}^i} \mu_{n})_{n\in\sigma(\N)}\xnarrowto{n\to\infty}\mu^i$ for every $i<k$,\footnote{Recall that $c_B \mu = (x\mapsto x-y)_\sharp(\mu \mres B)$ if $B = B_r(y)$ and $\mu \in \Mspace(\R^N)$.}
\item\label{bubbleSplitting} bubbles split: $\min_{0\le i<j< k_{n}}\dist(B_{n}^i,B_{n}^j) \xto{n\to\infty} +\infty$,
\item\label{bubbleDivergence} bubbles diverge: $ \min_{0\le i< k_n}\diam(B_{n}^i) \xto{n\to\infty} +\infty$,
\item\label{bubblingMassConservation} the bubbling mass is conserved: $\norm{\mu^b_{n}} \xto{\ell \to \infty} \sum_{0\le i<k} \norm{\mu^i}$,
\item\label{remainingVanishing} the remaining part is vanishing: $\sup_{x\in \R^N} \mu^v_{n}(B_1(x)) \xto{n\to\infty} 0$.
\end{enumerate}
\end{theorem}

Before proving \Cref{profileDecomposition}, we introduce the \enquote{bubbling} function of a sequence of finite \emph{signed} measures \((\mu_n)_{n\in\N}\):
\begin{equation}
m((\mu_n)_{n\in\N})\coloneqq\sup\Big\{ \norm{\mu} \st (\tau_{-x_{\sigma(\ell)}}\mu_{\sigma(\ell)})_{\ell\in\N} \mweakto \mu,\, \sigma \in \Sigma,\, x_{\sigma(\ell)} \in \R^N\, (\forall \ell) \Big\}.
\end{equation}
Although we will use this function on signed measures, we will start from a sequence of positive measures and use the following characterization of vanishing sequences, which holds only in the case of positive measures:
\begin{lemma}\label{concentrationProperties}
A sequence \((\mu_n)_{n\in\N}\) of finite positive measures over $\R^N$ is vanishing if and only if \(m((\mu_n)_{n\in\N})=0\).
\end{lemma}
\begin{proof}
Assume that \((\mu_n)_{n\in\N}\) is vanishing and that \((\tau_{-x_{\sigma(\ell)}}\mu_{\sigma(\ell)})_{\ell\in\N} \mweakto \mu\) for some \(\sigma \in \Sigma\) and some sequence of points \((x_{\sigma(\ell)})_{\ell\in\N}\). Then, for every \(x\in\R^N\),
\[
\mu(B_1(x))\le\liminf_{\ell\to\infty}\tau_{-x_{\sigma(\ell)}}\mu_{\sigma(\ell)}(B_1(x))=\liminf_{\ell\to\infty}\mu_{\sigma(\ell)}(B_1(x+x_{\sigma(\ell)}))=0,
\]
i.e. \(\mu=0\) and thus \(m((\mu_\ell)_{\ell\in\N})=0\).

Conversely, if \((\mu_n)_{n\in\N}\) is not vanishing, then there exists \(\eps>0\), \(\sigma\in\Sigma\) and a sequence of points \((x_n)_{n\in\sigma(\N)}\) in \(\R^N\) such that \(\mu_n(B_1(x_n))\ge\eps\) for every \(n\in\sigma(\N)\). Up to further extraction, one can assume that \((\tau_{-x_{\sigma(\ell)}}\mu_{\sigma(\ell)})_{\ell\in\N} \mweakto\mu\in\Mspace(\R^N)\). We have
\[
\mu(\bar B_1(0))\ge\limsup_{\ell\to\infty}\tau_{-x_{\sigma(\ell)}}\mu_{\sigma(\ell)}(\bar B_1(0))=\limsup_{\ell\to\infty}\mu_{\sigma(\ell)}(\bar B_1(x_{\sigma(\ell)}))\ge\eps>0,
\]
which entails \(m((\mu_\ell)_{\ell\in\N})\ge\eps>0\).
\end{proof}
%%%
\begin{proof}[Proof of \Cref{profileDecomposition}]
If \((\mu_n)_{n\in\N}\) is vanishing, then we take $\sigma = \id$ and $k=0$, so that $\mu_{\sigma(\ell)} = \mu_\ell = \mu^v_\ell$, \labelcref{bubbleEmergence,bubbleSplitting,bubbleDivergence,bubblingMassConservation} are empty statements and \ref{remainingVanishing} is satisfied since $(\mu_n)_{n\in\N}$ is vanishing. Assume on the contrary that \((\mu_n)_{n\in\N}\) is not vanishing. We shall construct the bubbles by induction and prove their properties in several steps.

%% Emergence %%%
%%%%%%%%%%%%%
\medskip
\noindent\emph{Step 1: construction of bubbles centers.} At first step (step $0$), since \(m((\mu_n)_{n\in\N})>0\), there exists $\sigma_0 \in \Sigma$ and a sequence of points \((x^0_{n})_{n \in \sigma_0(\N)}\), such that
\begin{equation}
(\tau_{-x^0_{n}} \mu_n)_{n \in \sigma_0(\N)}\mweakto\mu^0\in\Mspace(\R^N)\quad\text{with}\quad\norm{\mu^0} \geq \frac 12 m((\mu_n)_{n\in\N}).
\end{equation}
We then set \(\mu^0_n\coloneqq\mu_n-\tau_{x^0_n}\mu^0\) and we continue by induction, starting from the sequence \((\mu^0_n)_{n\in\sigma_0(\N)}\). More precisely, assume that for a fixed step $k-1 \in \N$, for every \(i\in\N\) such that \(0\le i\le k-1\), we have built $\mu^i \in \Mspace(\R^N)$, $\sigma_i \in \Sigma$, points $(x^i_n)_{n\in \sigma_i(\N)}$ and sequences $(\mu^i_{n})_{n\in \sigma_i(\N)} \in \Mspace(\R^N)$ such that for every $i$,
\begin{gather}
\sigma_i \preceq \sigma_{i-1},\label{bubbleInduction1}\\
\mu^i_{n} = \mu_n -\sum_{0\leq j \leq i} \tau_{x^j_n}\mu^j, \quad (\forall n \in \sigma_i(\N)),\label{bubbleInduction2}\\
(\tau_{-x^i_n} \mu^{i-1}_n)_{n\in\sigma_i(\N)} \mweakto \mu^i,\label{bubbleInduction3}\\
\norm{\mu^i} \geq \frac 12 m((\mu^i_n)_{n \in\sigma_i(\N)})>0,\label{bubbleInduction4}
\end{gather}
where $\sigma_{-1} \coloneqq \id, (\mu_n^{-1}) \coloneqq (\mu_n)$.  If $m((\mu^{k-1}_n)_{n\in \sigma_{k-1}(\N)}) = 0$, we stop; otherwise, we proceed to the next step $k$ to build $\sigma_k, \mu^k, (x_n^k)_{n\in\sigma_k(\N)}, (\mu_n^k)$ as we did at step $k=0$, starting with $(\mu^{k-1}_n)_{n\in\sigma_{k-1}(\N)}$. Either the induction stops at some step \(k-1\in\N\) for which \(m((\mu^{k-1}_{n})_{n\in \sigma_{k-1}(\N)})=0\) or the previous objects are defined for every \(i\in\N\), in which case we let \(k \coloneqq +\infty\).

%% Split %%%
%%%%%%%%%%%%
\medskip
\noindent\emph{Step 2: splitting of bubbles centers.} We prove that
\begin{equation}\label{splittingWeak}
\lim_{\sigma_i(\N) \ni n\to\infty} \dist(x^i_{n},x^j_{n}) =+\infty \quad\text{for every \(i,j\in\N\) with \(0\le j<i<k\)}.
\end{equation}
Indeed, assume by contradiction that there is a first index \(i<k\) such that for some \(j_0<i\), \((\dist(x^i_{n},x^{j_0}_{n}))_{n\in\sigma_i(\N)}\) is not divergent. In particular, there exists \(\sigma\preceq\sigma_i\) such that \((x^i_{n}-x^{j_0}_{n})_{n\in\sigma(\N)}\to x\in\R^N\). Moreover, \((\dist(x^i_n,x^j_n))_{n\in\sigma_i(\N)}\to\infty\), for every \(j<i\), \(j\neq j_0\) by minimality of $i$ and the triangle inequality $\dist(x_n^j,x_n^{j_0}) \leq \dist(x_n^j,x_n^i) + \dist(x_n^i,x_n^{j_0})$. Notice by \labelcref{bubbleInduction2} that for every $n\in \sigma(\N)$,
\begin{gather*}
\mu_n^{i-1} = \mu_n^{j_0-1} - \tau_{x_n^{j_0}} \mu^{j_0} - \sum_{j_0 < j < i} \tau_{x_n^j} \mu^j,\\
\shortintertext{hence taking the translation $\tau_{-x_n^i}$,}
\tau_{-x_n^i} \mu_n^{i-1} = \tau_{x_n^{j_0}-x_n^i} (\tau_{-x_n^{j_0}} \mu_n^{j_0-1}-\mu^{j_0}) - \sum_{j_0 < j < i} \tau_{x_n^j-x_n^i} \mu^j,\\
\shortintertext{and passing to the weak limit, knowing that $x_n^{j_0}-x_n^i \to -x$ and $\dist(x_n^j,x_n^i) \to +\infty$ for $j_0 < j <i$,}
\mu^i = \tau_{-x}(\mu^{j_0}-\mu^{j_0}) - \sum_{j_0 < j < i} 0 = 0.
\end{gather*}
This contradicts the fact that \((\tau_{-x^i_n}\mu^{i-1}_n)_{n\in\sigma(\N)}\mweakto\mu^i\neq 0\) and proves \labelcref{splittingWeak}.

%%%% Weak convergence %%%
%%%%%%%%%%%%%%%%%%%%%%%%%%%
\medskip
\noindent\emph{Step 3: weak convergence of bubbles.}
From \labelcref{bubbleInduction2} we get
\begin{equation}\label{bubbleWeak}
\tau_{-x^i_n}\mu^{i-1}_n =\tau_{-x^i_n}\mu_n-\sum_{0\le j<i}\tau_{-x^i_n+x^j_n}\mu^j,
\end{equation}
and by \labelcref{splittingWeak}, the sum converges weakly to $0$, and so
\begin{equation}\label{bubbleWeakConvergence}
(\tau_{-x^i_n}\mu_n)_{n\in\sigma_i(\N)}\mweakto \mu^i\quad\text{for every \(i\in\N\) with \(i<k\)}.
\end{equation}

%%% Construction of the balls with Mass conservation %%%%
%%%%%%%%%%%%%%%%%%%%%%%%%
\medskip
\noindent\emph{Step 4: construction of the bubbles with mass conservation.}
We now construct the extraction \(\sigma\in\Sigma\) that we need by induction: we set \(\sigma(0)=0\) and, assuming that \(\sigma(0)<\dots<\sigma(\ell-1)\), with \(\ell\in\N^\ast\), have been constructed, we set \(\sigma(\ell)\coloneqq n\) with \(n\in\sigma_{\ell \wedge k -1}(\N)\) large enough so that $n > \sigma(\ell-1)$ and for every \(i< \ell \wedge k\),
\begin{gather}
\mu_n(B_\ell(x^i_n))\le\norm{\mu^i}+2^{-\ell},\label{inductionBubbles1}\\
\shortintertext{and}
\min_{0\le j<i}\dist(x^i_n,x^j_n)\ge 4\ell.\label{inductionBubbles2}
\end{gather}
Such an \(n\) exists by \labelcref{splittingWeak} and \labelcref{bubbleWeakConvergence}, noticing that $\mu_n(B_\ell(x_n^i)) = (\tau_{-x_n^i} \mu_n)(B_\ell)$. Then for each \(n=\sigma(\ell)\), \(\ell\in\N\), we set \(k_n=\ell \wedge k\), and for each \(i\in\{0,\dots,k_n-1\}\), 
\[
B^i_n\coloneqq B_\ell(x^i_n).
\]
Finally, for every \(n\in\sigma(\N)\), we decompose $\mu_n$ as expected:
\[
\mu_n = \mu^b_n + \mu^v_n,\quad \text{where } \mu^b_n = \sum_{0\le i < k_n} \mu_n \mres B_n^i.
\]

Let us check the four first items \labelcref{bubbleEmergence}--\labelcref{bubblingMassConservation}. Notice that \labelcref{bubbleDivergence} is fulfilled because $\diam(B_{\sigma(\ell)}^i) = \ell \to +\infty$ as $\ell\to\infty$, and \labelcref{bubbleSplitting} because of \labelcref{inductionBubbles2}. Since for every $i<k$, $\lim_{\sigma(\N)\ni n \to\infty}\diam(B_n^i) = +\infty$ and $c_{B_n^i} \mu_n = (\tau_{-x_n^i} (\mu_n\mres B_n^i))$ for every $n \in \sigma_i(\N)$,  $(c_{B_n^i} \mu_n)_{n\in\sigma(\N)}$ converges weakly to $\mu^i$ by \labelcref{bubbleWeakConvergence}, and together with \labelcref{inductionBubbles1} it implies that
\[(c_{B_n^i} \mu_n)_{n\in\sigma(\N)} \narrowto \mu^i,\]
i.e. \labelcref{bubbleEmergence} is satisfied. 
Moreover, by \labelcref{inductionBubbles1} again,
\[
\limsup_{\ell\to\infty}\sum_{0\le i<k_{\sigma(\ell)}}\mu_{\sigma(\ell)}(B^i_{\sigma(\ell)})\le \sum_{0\le i<k}\norm{\mu^i}+\limsup_{\ell\to\infty}(\ell \wedge k) 2^{-\ell}=\sum_{0\le i<k}\norm{\mu^i},
\]
and since $k_n \to k$, by Fatou's lemma we have,
\[\sum_{0\leq i < k}\norm{\mu^i} \leq \liminf_{\ell\to\infty}\sum_{0\le i<k_{\sigma(\ell)}}\mu_{\sigma(\ell)}(B^i_{\sigma(\ell)}),\]
which proves \labelcref{bubblingMassConservation} because $ \sum_{0\le i<k_{\sigma(\ell)}}\mu_{\sigma(\ell)}(B^i_{\sigma(\ell)}) = \norm{\mu^b_{\sigma(\ell)}}$.

%%% Vanishing %%%%%
%%%%%%%%%%%%%%%%%%%
\medskip
\noindent\emph{Step 5: vanishing of the remaining part, proof of \labelcref{remainingVanishing}.} 
By \Cref{concentrationProperties}, it suffices to prove that \(m((\mu^v_n)_{n\in\sigma(\N)})=0\). We claim that:
\begin{equation}\label{bubbleConcentrationEstimate}
m((\mu^v_n)_{n\in\sigma(\N)})\le m((\mu^i_n)_{n\in\sigma_i(\N)}),\quad\text{for every \(i\in\N\) with \(i< k\)},
\end{equation}
which concludes since \(m((\mu^k_n)_{n\in\sigma_{k-1}(\N)})=0\) if \(k<\infty\), and \(m((\mu^i_n))_{n\in\sigma_i(\N)})\to 0\) as \(i\to\infty\) if \(k=\infty\). Indeed, if \(k=\infty\), we have by \labelcref{bubbleInduction4} and \labelcref{bubblingMassConservation},
\[
\frac 12\sum_{i\in\N}m((\mu^i_n)_{n\in\sigma_i(\N)})\le\sum_{i\in\N}\norm{\mu^i}= \lim_{\ell\to\infty} \norm{\mu_{\sigma(\ell)}^b} \leq \liminf_{\ell\to\infty}\norm{\mu_{\sigma(\ell)}}<\infty.
\]

Let us show \labelcref{bubbleConcentrationEstimate}. Let \(\bar\sigma\preceq\sigma\) and \((x_n)_{n\in\bar\sigma(\N)}\) be a sequence of points such that
\[
(\tau_{-x_n}\mu^v_n)_{n\in\bar\sigma(\N)}\mweakto \mu\in\Mspace(\R^N).
\]
We need to prove that \(\norm{\mu}\le m((\mu^i_n)_{n\in \sigma_i(\N)})\) for every $i <k$. Assume without loss of generality that \(\norm{\mu}>0\). Then for every \(i<k\),
\begin{equation}\label{distanceDiverge}
(\dist(x_n,x^i_n))_{n\in\bar\sigma(\N)}\to\infty.
\end{equation}
Otherwise, up to subsequence, $(\dist(x_n,x_n^i))_n$ would be bounded by some constant $M$, and for every $r > 0$,
\[(\tau_{-x_n} \mu_n^v)(B_r) \leq \mu_n^v(B_{r+M}(x^i_n)) \xto{n\to\infty} 0,\]
because \(\mu^v_n\) is supported on \(\R^N\setminus\cup_{0\le i<k_n}B^i_n\) and $B_{r+M}(x^i_n) \subseteq B^i_n$ for $n$ large enough by \labelcref{bubbleDivergence}. Hence \(\mu\) would be \(0\), a contradiction. Up to further extraction, one can assume that \((\tau_{-x_n}\mu_n)_{n\in\bar\sigma(\N)}\) converges weakly to a measure \(\bar\mu \in \Mspace(\R^N)\). Since \(\mu^v_n\le\mu_n\), we have \(\mu\le\bar\mu\). Moreover by \labelcref{bubbleInduction2}, for every \(i < k\) and \(n\in\ \bar\sigma(\N)\) large enough,
\[
\tau_{-x_n}\mu^i_n=\tau_{-x_n}\mu_n-\sum_{0\le j\leq i}\tau_{x^j_n-x_n}\mu^j,
\]
and because of \labelcref{distanceDiverge} the sum converges weakly to $0$, so that $\tau_{-x_n} \mu_n^i \mweakto \bar \mu$, and consequently,
\[
\norm{\mu}\le\norm{\bar\mu}\le m((\mu^i_n)_{n\in\sigma_i(\N)}),
\]
which is what had to be proved.
%% Recentering %%%%%
%%%%%%%%%%%%%%%%%%

\medskip
\noindent\emph{Step 6: re-centering of the bubbles at points of \(\supp\mu_n\).} 
By \labelcref{bubbleWeakConvergence}, $(\tau_{-x^i_{n}} \mu_n)_{n \in \sigma(\N)}$ converges weakly to the non-trivial measure $\mu_i$ for every $i<k$, thus
\begin{equation}
\label{finiteDistance}
R_i/2\coloneqq \limsup_{\sigma(\N) \ni n \to +\infty} \dist(\supp \mu_n, x_n^i) < +\infty.
\end{equation}
Therefore, for every $n$ large enough, there is a point $\tilde x_n^i$ such that $\abs{x_n^i-\tilde x_n^i} < R_i$ and $\tilde x_n^i \in \supp \mu_n$. After a further extraction, one may assume that for every $i$,  $\abs{x_n^i-\tilde x_n^i} < R_i < r_i^n$ where $\diam B_n^i = 2r_n^i$ for every $n$, and $(x_n^i -\tilde x_n^i)_{n\in \sigma(\N)}$ converges to some $p_i \in \R^N$. Finally, we set $\tilde r_i^n \coloneqq r_i^n-R_i$ and $\tilde B_n^i \coloneqq B(\tilde x_n^i,\tilde r_i^n) \subseteq B_n^i$. After replacing the balls $B_n^i$ by $\tilde B_n^i$, \labelcref{bubbleSplitting} and \labelcref{bubbleDivergence} are satisfied by definition. Notice that $(\tau_{-\tilde x^i_{n}} \mu_n)_{n \in \sigma(\N)}$ converges weakly to $\tilde \mu^i \coloneqq \tau_{p_i} \mu^i$ with $\norm{\tilde \mu^i} = \norm{\mu^i}$, and $\limsup_n \norm{c_{B_n^i} \mu_n} = \limsup_n \mu_n(\tilde B_n^i) \leq \limsup_n \mu_n(B_n^i) = \norm{\mu^i}$ hence \labelcref{bubbleEmergence} holds. Besides, using Fatou's lemma,
\begin{align*}
\limsup_n \sum_{i<k_n} \mu_n(\tilde B_n^i) &\leq \limsup_n \sum_{i<k_n} \mu_n(B_n^i)\\
&= \sum_{i<k} \norm{\mu^i}\leq \sum_{i<k} \liminf_n \mu_n(\tilde B_n^i) \leq \liminf_n \sum_{i<k_n} \mu_n(\tilde B_n^i)
\end{align*}
 so that $\lim_n \sum_{i<k_n} \mu_n(\tilde B_n^i) = \sum_i \norm{\mu_i}$ and \labelcref{bubblingMassConservation} is satisfied. In particular,
 \[\lim_n \sum_{i<k_n} \mu_n(B_n^i\setminus \tilde B_n^i) = \lim_n \sum_{i<k_n} \mu_n(B_n^i) - \lim_n \sum_{i<k_n} \mu_n(\tilde B_n^i) = 0,\]
and \labelcref{remainingVanishing} holds as well.
\end{proof}
\begin{remark}\label{reduceRadius}
If the sequence of families of balls $(B^i_n)_{0 \leq i < k_n}$ satisfies the conclusion of the theorem, i.e. \labelcref{bubbleEmergence}--\labelcref{remainingVanishing}, then it is also the case for any family of balls $(\tilde B^i_n)_{0 \leq i < k_n}$ with the same centers as those of \(B^i_n\) and with smaller but still divergent radii (i.e. satisfying \labelcref{bubbleDivergence}). It can be easily seen following the arguments at Step~6 of the proof.
\end{remark}

\subsection{Lower bound by concentration-compactness}

We will first establish a lower bound for the minimal energy along vanishing sequences defined on varying subsets of \(\R^N\). We say that a sequence of Borel functions \((u_n)_{n\in\N}\), each defined on some open set \(\Omega_n\subseteq\R^N\), is vanishing if the sequence of measures \((\abs{u_n}\lbm^N \mres \Omega_n)_{n\in\N}\) is vanishing in the sense of \Cref{DefinitionVanishing}, namely if \(\norm{u_n}_{\uloc({\Omega}_n)}\to 0\) as \(n\to\infty\), where \(\uloc({\Omega})\) is the set of uniformly locally integrable functions on the open set \(\Omega\), i.e. Borel functions $u$ on $\Omega$ such that
\begin{equation}
\norm{u}_{\uloc({\Omega})}\coloneqq\sup_{x\in \R^N} \int_{\Omega\cap (x+[0,1)^N)}\abs{u}<+\infty.
\end{equation}
It will be convenient to first extend our Sobolev functions to a neighbourhood \(\Omega_\delta\) of \(\Omega\) where for every \(\delta>0\) and every set \(X\subseteq\R^N\), we have set
\[
X_\delta\coloneqq\{x\in\R^N\st\dist(x,X)<\delta\}.
\]
We will need to consider sufficiently regular domains for which we have an extension operator \(W^{1,p}\cap \uloc (\Omega)\to W^{1,p}\cap \uloc(\Omega_\delta)\). We will only apply it to domains with smooth boundary, in which case we can use a reflection technique. Since we want quantitative estimates, we will use the notion of \emph{reach} of a set \(X\subseteq\R^N\) (see \cite{federerCurvatureMeasures1959}). We say that \(X\) has positive reach if there exists \(\delta>0\) such that every \(x\in X_\delta\) has a unique nearest point \(\pi(x)\) on \(X\). The greatest \(\delta\) for which this holds is denoted by \(\reach(X)\) and the map \(x\in X_{\reach(X)}\mapsto\pi(x)\in X\) is called the nearest point retraction. 
\begin{example}\label{perforatedDomain}
Assume that \(\Omega\) is a perforated domain \(B^0\setminus \bigcup_{i=1}^k B^i\) where the \(B^i\) are disjoint closed balls included in some open ball \(B^0\) (possibly \(B^0=\R^N\)). Then,
\[
\reach(\partial\Omega)=
\inf\{\operatorname{radius}(B^i)\st i=0,\dots,k\}
\cup\{\dist(\partial B^i,\partial B^j)\st i\neq j\}.
\]
\end{example}

By \cite[Theorem~4.8]{federerCurvatureMeasures1959}, we have
\begin{enumerate}[(i)]
\item\label{lipschitzProjection} if \(x,y\in X_\delta\) with \(0<\delta<\delta_0\coloneqq\reach(X)\), then
\(
\abs{\pi(x)-\pi(y)}\leq \frac{\delta_0}{\delta_0-\delta}\abs{x-y},
\)
\item\label{orthogonalProjection}
if \(x\in X\) and \(D_x\) is the intersection of \(X_{\reach(X)}\) with the straight line crossing \(\partial\Omega\) orthogonally at \(x\), then 
\(\pi(y)= x\) for every \(y\in D_x\).
\end{enumerate}

\begin{lemma}[Extension]\label{lemmaExtension}
Let \(\Omega\subseteq\R^N\) be an open set such that its boundary \(\partial\Omega\) is \(\Cspace^1\) with positive reach\footnote{Thanks to \cite[Remark~4.20]{federerCurvatureMeasures1959}, $\partial \Omega$ is actually of class $\Cspace^{1,1}$.}. Then, for every \(\delta\in (0,\reach(\partial\Omega))\), every \(p\in [1,+\infty)\) and every \(u\in L^1\cap W^{1,p}(\Omega)\), there exists \(\Bar{u}\in L^1\cap W^{1,p}(\Omega_\delta)\) such that \(\Bar{u}=u\) a.e. on \(\Omega\), and
\[
\norm{\Bar{u}}_{L^1(\Omega_\delta)}\leq A \norm{u}_{L^1(\Omega)},
\;\;
\norm{\Bar{u}}_{\uloc({\Omega}_\delta)}\leq A \norm{u}_{\uloc(\Omega)},
\;\;
\norm{\nabla\Bar{u}}_{L^p(\Omega_\delta)}\leq A\norm{\nabla u}_{L^p(\Omega)},
\]
with a constant \(A<+\infty\) depending only on \(N,\delta\) and \(\reach(\partial\Omega)\).
\end{lemma}
\begin{proof}
Let \(\sigma:(\partial\Omega)_\delta\to (\partial\Omega)_\delta\) be the reflection through \(\partial\Omega\), defined by \(\sigma(x)=2\pi(x)-x\). By the properties \labelcref{lipschitzProjection} and \labelcref{orthogonalProjection} of the nearest point retraction, we have that \(\sigma=\sigma^{-1}\) (simply because $\pi(\sigma(x)) = \pi(x))$ and $\sigma$ is \(L\)-Lipschitz with a constant \(L<+\infty\) depending on \(\delta\) and \(\reach(\partial\Omega)\) only.

We define\footnote{Note that \(\Bar{u}\) is not defined on \(\partial\Omega\), but this set is negligible.} \(\Bar{u}\) by \(\Bar{u}=u\) on \(\Omega\) and \(\Bar{u}=u\circ\sigma\) on \(\Omega_\delta\setminus{\Omega}\). This map is well-defined since \(\sigma(\Omega_\delta\setminus{\Omega})\subseteq\Omega\). Indeed, if we had \(x,\sigma(x)\in\Omega_\delta\setminus{\Omega}\), then the line segment \([x,\sigma(x)]\) would meet \(\partial\Omega\) orthogonally at its center \(\pi(x)\), and would remain out of \({\Omega}\), because otherwise there would exist a point $y$ belonging either to $\partial \Omega \cap (x,\pi(x))$ or $\partial \Omega \cap (\pi(x),\sigma(x))$ thus contradicting the definition of $\pi(x)$. Such a situation is not possible for a \(\Cspace^1\) boundary.

Moreover, by the change of variable formula and the chain rule, \(\Bar{u}\) satisfies the desired estimates since \(\sigma\) is bi-Lipschitz with its Lipschitz constants controlled in terms of \(\delta\) and \(\reach(\partial\Omega)\).
\end{proof}
We will need a localized version of the Gagliardo--Nirenberg--Sobolev inequality in a particular case:
\begin{lemma}\label{lemmaGagliardo}
Let \(\Omega\subseteq\R^N\) be an open set such that \(\partial\Omega\) is \(\Cspace^1\) with positive reach, let \(p\in [1,+\infty)\), let \(r\geq p(1+\frac{1}{N})\), and assume that \(r\leq\frac{pN}{N-p}\) when \(p<N\). Then for every \(u\in L^1\cap W^{1,p}(\Omega)\),
\[
\norm{u}_{L^r(\Omega)}\leq C \bigl(\norm{\nabla u}_{L^p(\Omega)}+\norm{u}_{L^1(\Omega)}\bigr)^\alpha \norm{u}_{\uloc({\Omega})}^{1-\alpha},
\]
%\(\alpha=\frac{1-\frac 1r}{1+\frac 1N-\frac 1p}\)
where \(\alpha\in (0,1]\) is the unique parameter such that \(\frac 1r=\alpha (\frac 1p-\frac 1N)+(1-\alpha)\), and the constant \(C<+\infty\) depends on \(N,r,p\) and \(\reach(\partial\Omega)\).
\end{lemma}
\begin{proof}[Proof of \Cref{lemmaGagliardo}]
We let \(u\in L^1\cap W^{1,p}(\Omega)\) and we extend \(u\) to \(\Bar{u}\in L^1\cap W^{1,p}(\Omega_\delta)\) as in \Cref{lemmaExtension}, with $\delta \coloneqq \reach(\Omega)/2$. By the Gagliardo--Nirenberg--Sobolev inequality (see \cite{nirenbergEllipticPartialDifferential1959}) on the hypercube \(Q_\delta=[-\frac{\delta}{\sqrt{N}},\frac{\delta}{\sqrt{N}})^N \subseteq \bar B_\delta\), we have for some $C$ depending on $N,\delta$,
\[
\norm{\Bar{u}}_{L^r(Q_\delta)}\leq C \norm{\nabla \Bar{u}}_{L^p(Q_\delta)}^\alpha\norm{\Bar{u}}_{L^1(Q_\delta)}^{1-\alpha}+C\norm{\Bar{u}}_{L^1(Q_\delta)}.
\]
We then cover \(\Omega\) with the disjoint hypercubes \(Q_\delta(c)=c+Q_\delta\subseteq\Omega_\delta\) centered at points \(c\) on the grid \(\Cspace\coloneqq\Omega\cap \frac{2\delta}{\sqrt{N}}\Z^N\). Since \(r \geq p(1+1/N)\), we can check that 
\begin{equation}\label{superadditiveExponent}
r\alpha=\frac{r-1}{1+\frac 1N-\frac 1p}\geq p.
\end{equation}
By superadditivity of \(s\mapsto s^{\frac{r\alpha}{p}}\) and of \(s\mapsto s^{r\alpha}\), we obtain
\begin{align*}
\norm{u}_{L^r(\Omega)}^r&\leq \sum_{c\in \Cspace}\norm{\Bar{u}}^r_{L^r(Q_\delta(c))}\\
&\leq 
C'\sum_{c\in \Cspace}\norm{\nabla \Bar{u}}_{L^p(Q_\delta(c))}^{p\frac{r\alpha}p}\norm{\Bar{u}}_{L^1(Q_\delta(c))}^{r(1-\alpha)}+C'\norm{\Bar{u}}^r_{L^1(Q_\delta(c))}\\
&\leq 
C'\norm{\nabla \Bar{u}}_{L^p(\Omega_\delta)}^{r\alpha}\norm{\Bar{u}}_{\uloc({\Omega}_\delta)}^{r(1-\alpha)}+C'\norm{\Bar{u}}_{L^1(\Omega_\delta)}^{r\alpha}\norm{\Bar{u}}_{\uloc({\Omega}_\delta)}^{r(1-\alpha)}\\
&\leq 
C''\bigl(\norm{\nabla u}_{L^p(\Omega)}+\norm{u}_{L^1(\Omega)}\bigr)^{r\alpha}\norm{u}_{\uloc({\Omega})}^{r(1-\alpha)}.\qedhere
\end{align*}
\end{proof}
\begin{proposition}\label{vanishingCost}
Assume that \(f:\R^N\times\R\times\R^N\to[0,+\infty]\) satisfies \labelcref{hypLowerContinuous} and \labelcref{hypCompact} for some \(p\in(1,+\infty)\). Consider a vanishing sequence $(u_n)_{n\in \N}$ in \(W^{1,1}_\loc(\Omega_n,\R_\pm)\), where the $\Omega_n\subseteq\R^N$ are open sets with \(\Cspace^1\) boundary such that \(\inf_{n\in\N}\reach(\partial\Omega_n)>0\), and a sequence $(\Phi_n)_{n\in\N}$ of Borel maps \(\Phi_n:\Omega_n\to\R^N\) such that $ \sup_{y\in\Omega_n}\abs{\Phi_n(y)-x_0} \to 0$ as \(n\to+\infty\) for some \(x_0\in\R^N\). If $\theta_n \coloneqq \int_{\Omega_n} u_n \neq 0$ for every $n$ and $(\theta_n)_{n\in\N}$ is bounded, then:
\[\liminf_{n\to+\infty} \frac 1{\abs{\theta_n}} \int_{\Omega_n} f(\Phi_n(y), u_n(y),\nabla u_n(y)) \dd y \geq f'_-(x_0,0^\pm,0),\]
where \(f'_-(x_0,0^\pm,0)\) was defined in \labelcref{slopeInf}.
\end{proposition}

\begin{proof}[Proof of \Cref{vanishingCost}]
Suppose for example that $u_n \geq 0$ a.e. for every $n$. Without loss of generality, we may assume after extracting a subsequence that:
\begin{equation}\label{boundEnergyMass}
K \coloneqq \sup_n \frac 1{\theta_n} \int_{\Omega_n} f(\Phi_n(y), u_n(y), \nabla u_n(y))\dd y + \theta_n < +\infty.
\end{equation}
For all $n\in\N$, we consider the measure $\nu_n\in \Mspace_+(\R^N \times \R \times \R^N)$ defined as the pushforward of the probability measure $\bar\mu_n = \frac 1{\theta_n}u_n \lbm^N \mres \Omega_n$ by the map $(\Phi_n, u_n,\nabla u_n)$, that is:
\[\nu_n \coloneqq (\Phi_n, u_n,\nabla u_n)_\sharp (\bar\mu_n).\]
We are going to show in several steps that $\nu_n \narrowto \delta_{(x_0,0,0)}$ and deduce the result. It suffices to show that the three projections $\nu^i_n \coloneqq (\pi^i)_\sharp \nu_n$, $i\in\{1,2,3\}$ converge narrowly to $\delta_{x_0}, \delta_0$ and $\delta_0$ respectively. Indeed, this would imply that $(\nu_n)$ converges narrowly to a measure concentrated on $(x_0,0,0)$, hence to $\delta_{(x_0,0,0)}$ since the $\nu_n$ are probability measures. First of all, since $(\nu_n)$ has bounded mass and $(\theta_n)$ is bounded, we may take a subsequence (not relabeled) such that $\nu_n \mweakto \nu $ and $\theta_n \to \theta$ as \(n\to\infty\) for some $\nu \in \Mspace_+(\R^N \times \R\times \R^N)$ and $\theta \geq 0$.

\medskip
\noindent\emph{Step~1: $\nu_n^1 \narrowto \delta_{x_0}$.} This is a direct consequence of the fact that $\nu_n^1$ is concentrated on $\Phi_n(\R^N)$ for every $n$ and $\dist(\Phi_n(\R^N),x_0) \to 0$ as \(n\to\infty\).

\medskip
\noindent\emph{Step~2: $\nu_n^2 \narrowto \delta_{0}$.} By \labelcref{boundEnergyMass} and our assumption \labelcref{hypCompact}, there is a constant \(K_1>0\) with
\begin{equation}
\label{nonHomogeneousSobolev}
\int_{\Omega_n}\abs{\nabla u_n}^p\leq K_1\int_{\Omega_n}u_n,\quad n\in\N.
\end{equation}
We deduce from Markov's inequality, and \Cref{lemmaGagliardo} applied with \(r=p(1+\frac 1N)\), corresponding to \(\alpha=\frac{N}{N+1}\), that
\begin{align*}
\nu^2_n([\eta,+\infty)) &= \frac 1{\theta_n} \int_{\{u_n \geq \eta\}} u_n\\
&= \frac 1{\theta_n} \int_{\{u_n \geq \eta\}} u_n^{1-r}u_n^r
\\
&\leq \frac{1}{\theta_n\eta^{r-1}}\int_{\Omega_n}u_n^r\\
&\leq \frac{C}{\theta_n\eta^{r-1}} \bigl(\norm{\nabla u_n}_{L^p(\Omega_n)}+\norm{u_n}_{L^1(\Omega_n)}\bigr)^{r\alpha} \norm{u_n}_{\uloc({\Omega}_n)}^{r(1-\alpha)}\\
&\leq \frac{C'}{\eta^{r-1}} \bigl(1+\theta_n^{p-1}\bigr)\norm{u_n}_{\uloc({\Omega}_n)}^{r(1-\alpha)},
\end{align*}
where in the last inequality, we have used the identity \(\alpha r=p\) and \labelcref{nonHomogeneousSobolev}, and $C,C'$ depend only on $N,r,p$ and $\inf_n \reach(\partial \Omega_n)$.

Since \((u_n)_{n\in\N}\) is vanishing and \((\theta_n)_{n\in\N}\) is bounded, the last term in the previous inequality goes to zero as \(n\to\infty\) and it follows that $\nu_n^2 \narrowto \delta_0$.

\medskip
\noindent\emph{Step~3: $\nu_n^3 \narrowto \delta_{0}$.} Fix $M > 0$ and $\eta > 0$. One has by \labelcref{nonHomogeneousSobolev},
\begin{align*}
\nu_n^3([M,+\infty)) = \frac 1{\theta_n} \int_{\{\abs{\nabla u_n} \geq M\}} u_n &\leq \frac 1{\theta_n} \int_{\{u_n < \eta\}\cap\{\abs{\nabla u_n} \geq M\}} u_n + \frac 1{\theta_n} \int_{\{u_n > \eta\}} u_n\\
&\leq \frac \eta{\theta_n} \lbm^N(\{\abs{\nabla u_n} \geq M\}) + \nu_n^2([\eta,+\infty))\\
&\leq \frac\eta{\theta_n} \frac 1{M^p} \int_{\Omega_n} \abs{\nabla u_n}^p + \nu_n^2([\eta,+\infty))\\
&\leq \frac{\eta K_1}{M^p}+ \nu_n^2([\eta,+\infty)).
\end{align*}
By the previous step, we know that $\lim_{n\to+\infty} \nu_n^2([\eta,+\infty)) = 0$, hence taking the superior limit as $n\to+\infty$ then $\eta \to 0$ we get
\(\lim_{n\to+\infty} \nu_n^3([M,+\infty)) = 0\).
Since this is true for every $M > 0$ we obtain $\nu_n^3 \narrowto \delta_0$.

\medskip
\noindent\emph{Step~4: conclusion.} By the previous steps, we deduce that $\nu_n \narrowto \delta_{(x_0,0,0)}$ as $n\to+\infty$. We define $g : \R^N \times \R_+ \times \R^N \to [0,+\infty]$ as the lower semicontinuous envelope of $\R^N \times \R_+^* \times \R^N \ni (x,u,\xi) \mapsto \frac 1u f(x,u,\xi)$. By \labelcref{hypLowerContinuous}, we have \(g(x,u,\xi)=\frac 1u f(x,u,\xi)\) if \(u>0\), and by definition of $f'_-$ (see \labelcref{slopeInf}), we have \(g(x,0,0)=f'_-(x,0^+,0)\) for every \(x\in\R^N\). Hence, by lower semicontinuty of $g$ and weak convergence of $(\nu_n)$, we get
\begin{align*}
\liminf_{n\to\infty} \int_{\Omega_n} f(\Phi_n,u_n,\nabla u_n) &\geq \liminf_{n\to\infty} \int_{\{u_n>0\}} \frac{f(\Phi_n,u_n,\nabla u_n)}{u_n} u_n\\
 &= \liminf_{n\to\infty} \int_{\R^N \times \R \times \R^N} g(x,u,\xi) \dd \nu_n(x,u,\xi)\\
&\geq \int_{\R^N} g(x,u,\xi) \dd \delta_{(x_0,0,0)}=f'_-(x_0,0^+,0),
\end{align*}
which ends the proof of the lemma.
\end{proof}

As a corollary, we may now relate the slope at $0$ of $H_f$ to that of $f$.

\begin{corollary}
Assume that $f : \R^N \times \R \times \R^N \to [0,+\infty]$ satisfies \labelcref{hypLowerContinuous}, \labelcref{hypCompact} for some \(p\in(1,+\infty)\) and \labelcref{hypSlope}. Fix $x\in \R^N$. If either $N \geq 2$ or ($N=1$ and $H_f(x,\cdot)\not\equiv +\infty$ on $\R_\pm^\ast$), then $H_f'(x,0^\pm) = f'_-(x,0^\pm,0)$.
\end{corollary}
\begin{proof}
The inequality $H_f'(x,0^\pm) \leq f'_-(x,0^\pm,0)$ is precisely \labelcref{hypSlope}, and the converse inequality $H_f'(x,0^\pm) \geq f'_-(x,0^\pm,0)$ comes from \Cref{vanishingCost}. Indeed, if $(u_n)_{n\in\N} \in W^{1,1}_\loc(\R^N,\R_\pm)$ is a sequence of functions of mass $\theta_n = \int_{\R^N} u_n$ going to $0$ and which is almost minimizing in the sense that $\lim_{n\to\infty} \frac{\ener^x_f(u_n)}{\abs{\theta_n}} = \liminf_{n\to\infty} \frac{H(x,\theta_n)}{\theta_n}$ then $(u_n)_{n\in\N}$ is vanishing and \Cref{vanishingCost} yields
\[\liminf_{n\to\infty} \frac{\ener^x_f(u_n)}{\abs{\theta_n}} \geq f'_-(x,0^\pm,0).\]
\end{proof}

We now establish our main energy lower bound along sequences with bounded mass (not necessarily vanishing):
\begin{proposition}\label{almostMinimisers}
Assume that $(f_\eps)_{\eps > 0}$ is a family of functions \(f_\eps:\R^N\times\R\times\R^N\to[0,+\infty]\) satisfying \labelcref{hypLowerContinuous}, \labelcref{hypConvex}, \labelcref{hypCompact} and \labelcref{hypGammaCv} where \(f = \lim_\eps f_\eps\). Let \((\eps_n)_{n\in\N}\) be a sequence of positive numbers going to zero, \((R_n)_{n\in\N}\) and \((r_n)_{n\in\N}\) be two sequences in \((0,+\infty]\) such that \(\lim_{n\to\infty}r_n=\lim_{n\to\infty} R_n-r_n= +\infty\), \((u_n)_{n\in\N}\) be a sequence of functions \(u_n \in W^{1,1}_\loc(B_{R_n},\R_\pm)\) with finite limit mass $m \coloneqq \lim_{n\to\infty} \int_{B_{r_n}} u_n$, and $(\Phi_n)_{n\in\N}$ be a sequence of Borel maps \(\Phi_n:B_{R_n}\to\R^N\) such that 
\begin{equation}
\label{hypPhiN}
\sup_{y\in B_{R_n}}\abs{\Phi_n(y)-x_0} \xto{n\to\infty} 0\quad \text{for some \(x_0\in\R^N\).}
\end{equation}
Then there exists a family \((u^i)_{0\leq i<k}\) of functions in \(W^{1,1}_\loc(\R^N,\R_\pm)\) with \(k\in\N\cup\{+\infty\}\), such that \(m_i\coloneqq\int_{\R^N}u^i\in\R_\pm^\ast\) for every \(i\), and
\begin{gather}\label{MassVanishingFormula}
m=m_v+\sum_{0\leq i<k}m_i\quad\text{with \(\pm m_v\geq 0\),}\\
\label{decompositionSequences}
\liminf_{n\to\infty}\int_{B_{R_n}}f_{\eps_n}(\Phi_n,u_n,\nabla u_n) \geq \abs{m_v} f'_-(x_0,0^\pm,0)+\sum_{0\leq i<k}\int_{\R^N}f(x_0,u^i,\nabla u^i).
\end{gather}
\end{proposition}

\begin{proof}
Suppose for example that $u_n \geq 0$ a.e. for every $n$. We first assume, up to subsequence, that the left hand side of \labelcref{decompositionSequences} is a finite limit. We apply the profile decomposition \Cref{profileDecomposition} to the sequence of positive measures \(\mu_n=u_n \lbm^N \mres B_{r_n}\) where, we assume the extraction \(\sigma\) to be the identity for convenience, and we use the same notation as in \Cref{profileDecomposition}. In particular, for each bubble \(B^i_n=B_{r^i_n}(x^i_n)\), with \(0\leq i<k_n\), we have \(x^i_n\in\supp\mu_n\subseteq \bar B_{r_n}\). By assumption, we have \(\lim_{n\to\infty}(R_n-r_n)=+\infty\); hence, up to reducing the radii of the balls \(B^i_n\) if necessary, in such a way that their radii still diverge (see \Cref{reduceRadius}), we can assume that
\begin{equation}
\label{bubbleStrongInclusion}
B^i_n\subseteq B_{R_n-1},\quad 0\leq i<k_n.
\end{equation}

For each \(0\leq i<k_n\), we let \(u^i_n\coloneqq u_n(\cdot+x^i_n)\). Since \labelcref{decompositionSequences} is assumed to be finite, we get that the sequence \((u^i_n)_n\) is bounded in \(W^{1,p}_\loc(\R^N)\) by \labelcref{hypCompact}. Hence, after a further extraction if needed, we get that \((u^i_n)_{n\in\N}\weakto u^i\) weakly in \(W^{1,p}_\loc(\R^N)\) for some limit \(u^i\), for every \(0\leq i<k=\lim k_n\). Setting \(m_i=\int_{\R^N}u^i\) for every $i$, by \labelcref{bubblingMassConservation} in \Cref{profileDecomposition}, we have
\[
m_v\coloneqq m-\sum_{0\leq i<k}m_i=\lim_{n\to\infty}\int_{B_{r_n}\setminus\cup_{0\leq i<k_n}B^i_n}u_n.
\]

Fix \(\eps>0\). We decompose the energy as
 \begin{multline}\label{lowerBoundExistence}
\int_{B_{R_n}}f_{\eps}(\Phi_n,u_n,\nabla u_n)
=
\int_{B_{R_n}\setminus \cup_{0\leq i<k_n}B^i_n}f_{\eps}(\Phi_n,u_n,\nabla u_n)\\
+
\sum_{0\leq i<k_n}\int_{B_{r^i_n}}f_{\eps}(\Phi_n(\cdot+x^i_n),u^i_n,\nabla u^i_n).
\end{multline}
Note that the domains \(\Omega_n\coloneqq B_{R_n}\setminus \cup_{0\leq i<k}B^i_n\) satisfy \(\inf_{n\in\N}\reach(\partial\Omega_n)>0\) as noticed in \Cref{perforatedDomain}, thanks to \labelcref{bubbleStrongInclusion} and \labelcref{bubbleSplitting}, \labelcref{bubbleDivergence} in \Cref{profileDecomposition}. Hence, applying \Cref{vanishingCost} to the Lagrangian \(f_\eps\), we obtain
\begin{equation}
\label{vanishingLowerBoundExistence}
\liminf_{n\to\infty}\int_{B_{R_n}\setminus \cup_{0\leq i<k_n}B^i_n}f_{\eps}(\Phi_n,u_n,\nabla u_n)
\geq m_v(f_\eps)'_-(x_0,0^+,0).
\end{equation}
 Moreover, by the lower semicontinuity of integral functionals (see \cite[Theorem~4.1.1]{buttazzoSemicontinuityRelaxationIntegral1989}), in view of \labelcref{hypPhiN}, we have for each \(i\) with \(0\leq i<k\),
\begin{equation}\label{lowerBoundButtazzo}
\liminf_{n\to\infty}
\int_{B_{r^i_n}}f_{\eps}(\Phi_n(\cdot+x^i_n),u^i_n,\nabla u^i_n)
\geq
\int_{\R^N}f_\eps(x_0,u^i,\nabla u^i).
\end{equation}
Finally, by \labelcref{lowerBoundExistence}, \labelcref{vanishingLowerBoundExistence}, \labelcref{lowerBoundButtazzo}, \labelcref{hypGammaCv} together with monotone convergence, we deduce that
\begin{align*}
&\quad\liminf_{n\to\infty}\int_{B_{R_n}}f_{\eps_n}(\Phi_n,u_n,\nabla u_n)\\
&\geq
\lim_{\eps\to 0^+}\Bigl(
m_v (f_\eps)'_-(x_0,0^+,0)
+
\sum_{0\leq i<k}
\int_{\R^N}f_\eps(x_0,u^i,\nabla u^i)\Bigr)
\\
&=m_vf'_{-}(x_0,0^+,0)
+
\sum_{0\leq i<k}
\int_{\R^N}f(x_0,u^i,\nabla u^i).
\end{align*}
The similar statement for non-positive functions is obtained in the same way.
\end{proof}
\subsection{Existence of optimal profiles}

For the existence of an optimal profile in \labelcref{def_Hf}, we need a criterion that rules out splitting and vanishing of minimizing sequences:
\begin{lemma}\label{IndecomposableLinear}
Let \(H:\R_+\to\R_+\) be a concave function. Then $H$ is subadditive, and if for some \(0 < \theta < m\) one has \(H(m)=H(m-\theta)+H(\theta)\), then \(H\) is linear on \((0,m)\).
\end{lemma}
\begin{proof} By concavity, \(t\mapsto \frac{H(t)}{t}\) is non-increasing. Hence,
\begin{align*}
H(m)=\theta\frac{H(m)}{m}+(m-\theta)\frac{H(m)}{m}\leq \theta\frac{H(\theta)}{\theta}+(m-\theta)\frac{H(m-\theta)}{m-\theta}.
\end{align*}
But, by assumption, the last inequality is an equality which means that \(\frac{H(m)}{m}=\frac{H(\theta)}{\theta}=\frac{H(m-\theta)}{m-\theta}\). In particular, the monotone function \(t\mapsto \frac{H(t)}{t}\) must be constant on \([\theta,m]\), i.e. \(H\) must be linear on \([\theta,m]\). By concavity this is only possible if \(H\) is linear on \([0,m]\).
\end{proof}

We can now state and prove our existence result:
\begin{proposition}\label{existenceProfile}
Assume that $f:\R \times \R^N \ni (u,\xi) \mapsto f(u,\xi)\in [0,+\infty]$ satisfies \labelcref{hypLowerContinuous}, \labelcref{hypConvex}, \labelcref{hypZero}, \labelcref{hypCompact} and \labelcref{hypSlope}. Let \(m\in \R_+\) (resp. \(m\in \R_-\)). If the cost function $H_f$, defined in \labelcref{PropertiesDefinitionH}, is not linear on \([0,m]\) (resp. $[m,0]$), then the minimization problem in \labelcref{PropertiesDefinitionH} admits a solution \(u\in W^{1,1}_\loc(\R^N)\), i.e. \(\int_{\R^N} u = m\) and \(\int_{\R^N}f(u,\nabla u)=H_f(m)\), such that $u\geq 0$ (resp. $u\leq 0$) in $\R^N$.
\end{proposition}
\begin{proof}
We consider the case $m\geq 0$, the case $m<0$ can then be deduced by considering $\tilde{f}(u,\xi)=f(-u,-\xi)$. We assume without loss of generality that $H_f$ is finite on $(0,+\infty)$, otherwise by \Cref{concaveDN} there is nothing to prove. By \Cref{without_f_zero_at_zero}, the admissible class in \labelcref{PropertiesDefinitionH} can be reduced to non-negative functions. In particular, if \(m=0\), then \(u=0\) is the only non-negative solution. If \(m>0\), we apply \Cref{almostMinimisers} in the following situation: \(f_\eps(x,u,\xi)=f(u,\xi)\) for every \((x,u,\xi)\in\R^N\times\R\times\R^N,\varepsilon>0\), \(R_n\equiv +\infty\), \(\Phi_n\equiv x_0\in\R^N\), \((u_n)_{n\in\N}\) is a minimizing sequence for the minimization problem in \labelcref{PropertiesDefinitionH}, and \((r_n)_{n\in\N}\) is a sequence of positive radii going to \(+\infty\) such that \(\lim_{n\to\infty}\int_{B_{r_n}}u_n=m\). We obtain 
\[
H_f(m)\geq m_v f'_-(0^+,0)+\sum_{0\leq i<k}\int_{\R^N}f(u^i,\nabla u^i),
\]
with \(k\in\N\cup\{+\infty\}\), \(u^i\in W^{1,p}_\loc(\R^N,\R_+)\) and \(m=\sum_{0\leq i<k}m_i+m_v\), where \(m_i\coloneqq\int_{\R^N}u^i\). By \Cref{slopeUpperBound} and \Cref{slopeUpperBound1D}, in view of our assumption \labelcref{hypSlope}, and since $H_f$ is assumed to be finite on $(0,+\infty)$ (for the case $N=1$), we have \(f'_-(0^+,0)\geq H_f'(0^+)\). Moreover, by \Cref{concaveDN}, we have \(m_v H_f'(0^+)\geq H_f(m_v)\). Hence, by definition of \(H_f\),
\[
H_f(m)\geq H_f(m_v)+\sum_{0\leq i<k}H_f(m_i).
\]
Since the concave function \(H_f\) is not linear on \([0,m]\), by \Cref{IndecomposableLinear}, we have either \(k=1\) and \(m_v=0\), and we are done, or \(k=0\) and \(m=m_v\). But in the latter case, we would have \(H_f(m)=mH_f'(0^+)\) which implies that the monotone function \(t\mapsto \frac{H_f(t)}{t}\) is constant on \([0,m]\), i.e. that \(H_f\) is linear on \([0,m]\). This contradicts our assumption.
\end{proof}
\begin{remark}\label{compact_modulo}
Notice that the end of the proof actually shows, under the given assumptions, that the set of minimizers for a given mass $m$ is compact in $L^1$ modulo translations.
\end{remark}

\section{\texorpdfstring{$\Gamma$}{Γ}-convergence of the rescaled energies towards the \texorpdfstring{\(H\)-mass}{H-mass}}\label{sectionGammaConvergence}

We establish lower and upper bounds for the $\Gamma-\liminf$ and $\Gamma-\limsup$ respectively, from which we deduce the proof of our main $\Gamma$-convergence result. The upper bound on the $\Gamma-\limsup$ holds under more general assumptions and will be needed in \Cref{droplets}.

\subsection{Lower bound for the \texorpdfstring{\(\Gamma-\liminf\)}{Γ-liminf}}

Given a Borel function $f : \R^N \times \R \times \R^N \to [0,+\infty]$, we define for every $(x,m)\in \R^N\times \R$,
\begin{equation}
\label{definitionLowerCost}
H_f^-(x,m) \coloneqq H_f(x,m) \wedge (f'_-(x,0^\pm,0) \abs{m}), \quad\text{if $\pm m\geq 0$},
\end{equation}
recalling that \(H_f\) is defined in \labelcref{def_Hf} and \(f_{-}'(x,0^\pm,0)\) in \labelcref{slopeInf}, with the usual convention $(\pm\infty) \times 0 = 0$. Notice that it is concave on $\R_+$ and $\R_-$ by \Cref{concaveDN}, and under \labelcref{hypSlope} we have \(H_f^-(x,m) =H_f(x,m)\).

\begin{proposition}\label{lowerBound}
Assume that $(f_\eps)_{\eps > 0}$ is a family of functions \(f_\eps:\R^N\times\R\times\R^N\to[0,+\infty]\) satisfying \labelcref{hypLowerContinuous}, \labelcref{hypConvex}, \labelcref{hypZero}, \labelcref{hypCompact} and \labelcref{hypGammaCv} where $f=\lim_{\eps\to 0} f_\eps$. Let \((\eps_n)_{n\in\N}\) be a sequence of positive numbers going to zero, \((u_n)_{n\in\N}\) be a sequence in \(W^{1,1}_\loc(\R^N)\), and let 
\[
e_n \coloneqq f_{\eps_n}(\cdot,\eps_n^N u_{n},\eps_n^{N+1} \nabla u_n)\eps_n^{-N} \lbm^N
\]
be the energy measure associated with $u_n$. If \(u_n\lbm^N\mweakto \mu\in\Mspace(\R^N)\) and $e_n \mweakto e \in \Mspace(\R^N)$, then
\begin{equation}
\label{mainLowerBound}
e \geq H_f^-(\mu).
\end{equation}
In particular, \(\Gamma(\Cspace_0')-\liminf_{\eps \to 0} \ener_\eps \geq \mass^{H_f^-}\).
\end{proposition}

\begin{proof}[Proof of \Cref{lowerBound}]
Set $H \coloneqq H_f^-$ and recall that it is concave on $\R_+$ and $\R_-$ by \Cref{concaveDN}. Let us assume first that $u_n \geq 0$ a.e. for every $n$. To obtain \labelcref{mainLowerBound}, it is enough to prove that for every \(x_0\in\R^N\),
\begin{equation}\label{atomic}
e(\{x_0\})\ge H(x_0,\mu(\{x_0\})).
\end{equation}
and that if \(x_0 \in \supp \mu\) is not an atom of $\mu$, then
\begin{equation}\label{diffuse}
\limsup_{R\to 0^+}\frac{e(B_R(x_0))}{\mu(B_R(x_0))}\ge H'(x_0,0^+),
\end{equation}
Indeed \labelcref{atomic} implies that $e \geq (H(\mu))^a$ (the atomic part of the measure \(H(\mu)\) defined in \Cref{defHmass}) while \labelcref{diffuse} implies that $e \geq H'(\cdot,0^+) \mu^d = (H(\mu))^d$, by Radon-Nikod\`ym theorem (see \cite[Theorem~2.22]{ambrosioFunctionsBoundedVariation2000}); these two relations yield \(e\geq (H(\mu))^a+(H(\mu))^d=H(\mu)\) as required.

\medskip

We fix $x_0 \in \supp \mu$ and proceed in several steps.

\medskip
\noindent\emph{Step 1: blow-up near \(x_0\).} We first take two sequences of positive radii \((R_\ell)_{\ell\in\N}\to 0\) and \((r_\ell)_{\ell\in\N}\) such that for every \(\ell\in\N\), \(r_\ell\in (0,R_\ell)\),
\begin{gather}
e(\partial B_{R_\ell}(x_0))=\mu(\partial B_{r_\ell}(x_0))=0,\label{seqBoundary}\\
\shortintertext{and}
\lim_{\ell\to\infty}\frac{e(B_{R_\ell}(x_0))}{\mu(B_{r_\ell}(x_0))}
=\limsup_{R\to 0^+}\frac{e(B_R(x_0))}{\mu(B_R(x_0))}.\label{seqUpperDensity}
\end{gather}
This last property is obtained by taking first a sequence $(\rho_\ell)_\ell$ such that 
\[
\limsup_{R\to 0^+}\frac{e(B_R(x_0))}{\mu(B_R(x_0))} = \lim_{\ell\to\infty} \frac{e(B_{\rho_\ell}(x_0))}{\mu(B_{\rho_\ell}(x_0))},
\]
then using monotone convergence the measures to get first \(r_\ell\) then \(R_\ell\) such that $0 < r_\ell < R_\ell < \rho_\ell$, \(\mu(B_{r_\ell}(x_0))\geq (1-2^{-\ell})\mu(B_{\rho_\ell}(x_0))\) and \(e(B_{R_\ell}(x_0))\geq (1-2^{-\ell})e(B_{\rho_\ell}(x_0))\).

By weak convergence and \labelcref{seqBoundary}, according to \cite[Proposition~1.62~b)]{ambrosioFunctionsBoundedVariation2000}, we have for every \(\ell\in\N\),
\[
\lim_{n\to\infty}e_n(B_{R_\ell}(x_0))=e(B_{R_\ell}(x_0))\quad\text{and} \quad \lim_{n\to\infty}\int_{B_{r_\ell}(x_0)}u_n=\mu(B_{r_\ell}(x_0)).
\]
Hence, there exists an extraction \((n_\ell)_{\ell\in\N} \in \Sigma\) such that 
\begin{equation}\label{radiiCovering}
\lim_{\ell\to\infty}\frac{r_\ell}{\eps_{n_\ell}}=+\infty \quad \text{and}\quad \lim_{\ell\to\infty}\frac{R_\ell-r_\ell}{\eps_{n_\ell}}=+\infty,
\end{equation}
satisfying the following conditions:
\begin{gather}
\label{atomsMassEnergy}
\mu(\{x_0\})=\lim_{\ell\to\infty}\int_{B_{r_\ell}(x_0)}u_{n_\ell}, 
\quad
e(\{x_0\}) = \lim_{\ell\to\infty} e_{n_\ell}(B_{R_\ell}(x_0)),
\\
\shortintertext{and}
\label{densityExtraction}
\limsup_{\ell\to\infty}\frac{e(B_{R_\ell}(x_0))}{\mu(B_{r_\ell}(x_0))}
=\lim_{\ell\to\infty}\frac{e_{n_\ell}(B_{R_\ell}(x_0))}{\int_{B_{r_\ell}(x_0)}u_{n_\ell}}.
\end{gather}

We may rewrite the mass and energy in terms of the re-scaled map \(v_\ell\) defined by
\begin{equation}
v_\ell(y)\coloneqq \eps_{n_\ell}^N u_{n_\ell}(x_0+\eps_{n_\ell} y),\quad y\in\R^N, \ell\in\N,
\end{equation}
as follows:
\begin{gather}\label{expressionRescaledMass}
\int_{B_{r_\ell}(x_0)}u_{n_\ell}=\int_{B_{\eps_{n_\ell}^{-1}r_{\ell}}}v_\ell,\\
\shortintertext{and}
\label{expressionRescaledEnergy}
e_{n_\ell}(B_{R_\ell}(x_0))=\int_{B_{\eps_{n_\ell}^{-1}R_{\ell}}}f_{\eps_{n_\ell}}(x_0+\eps_{n_\ell} y,v_{\ell}(y),\nabla v_{\ell}(y))\dd y.
\end{gather}

\noindent\emph{Step 2: proof of \labelcref{atomic}.} By \Cref{almostMinimisers}, we have
\begin{equation}
\label{lowerEnergyBound}\begin{aligned}
e(\{x_0\}) &= \lim_{\ell\to\infty} \int_{B_{\eps_{n_\ell}^{-1}R_{\ell}}}f_{\eps_{n_\ell}}(x_0+\eps_{n_\ell} y,v_{\ell}(y),\nabla v_{\ell}(y))\dd y\\
&\geq m_v f'_{-}(x_0,0^+,0)+\sum_{0\leq i<k} H_f(x_0,m_i).\end{aligned}
\end{equation}
Here \(k\in\N\cup\{+\infty\}\) and \(m=m_v+\sum_{0\leq i<k}m_i\), with  \(m_i> 0\), \(m_v\geq 0\) and
\[
m=\lim_{\ell\to\infty} \int_{B_{\eps_{n_\ell}^{-1}r_{\ell}}}v_\ell =\mu(\{x_0\}).
\]

Since the function \(H=H_f^-\), defined in \labelcref{definitionLowerCost}, is the infimum of two functions which are mass-concave, it is mass-concave hence subadditive. From \labelcref{lowerEnergyBound} we thus arrive at
\[
e(\{x_0\})\geq H(x_0,m_v) + \sum_{0\leq i<k}H(x_0,m_i) \geq H\Bigl(x_0,m_v+\sum_{0\leq i<k}m_i\Bigr)=H(x_0,\mu(\{x_0\})).
\]

\medskip
\noindent\emph{Step 3: proof of \labelcref{diffuse}.} Fix $\eps >0$ and assume that \(m=\mu(\{x_0\})=0\). In that case, we apply \Cref{vanishingCost} to the sequence of functions $(v_\ell)_{\ell\in\N}$ defined on the sets $\Omega_\ell = B_{\eps_{n_\ell}^{-1}r_\ell}$ and the function $f_\eps$ to get, thanks to \labelcref{hypGammaCv}:

\begin{align*}
\limsup_{R\to 0^+}\frac{e(B_R(x_0))}{\mu(B_R(x_0))} &= \lim_{\ell\to \infty} \frac{e_{n_\ell}(B_{R_\ell}(x_0))}{\int_{B_{r_\ell}(x_0)} u_{n_\ell}}\\
&\geq \liminf_{\ell\to\infty}\frac{1}{\int_{B_{\eps_{n_\ell}^{-1}r_\ell}} v_{\ell}}\int_{B_{\eps_{n_\ell}^{-1}r_\ell}} f_{\eps}(x_0+\eps_{n_\ell}y,v_\ell(y),\nabla v_\ell(y)) \\
&\geq (f_\eps)_-'(x_0,0^+,0).
\end{align*}
Taking the limit $\eps \to 0^+$, we deduce by \labelcref{hypGammaCv} and \labelcref{definitionLowerCost}:
\begin{equation}\label{vanishingDiffuse}
\limsup_{R\to 0^+}\frac{e(B_R(x_0))}{\mu(B_R(x_0))} \geq f_-'(x_0,0^+,0) \geq H'(x_0,0^+).
\end{equation}
In view of the discussion at the beginning of the proof, we have now proved \labelcref{mainLowerBound}. 

\medskip
\noindent\emph{Step 4: proof of \labelcref{mainLowerBound} for signed $(u_n)_n$.} Notice that the preceding reasoning for non-negative $u_n$ applies also to the case of non-positive $u_n$. Let us handle the case where the $(u_n)$'s may change sign. We simply apply the above cases to the positive and negative parts $((u_n)_\pm)_n$ which converge weakly as measures (up to subsequence) to some measures $\mu^\pm \in \Mspace_+(\R^N)$ which satisfy $\mu = \mu^+ - \mu^-$, so that $e\geq H(\pm \mu^\pm)$. We know that the Jordan decomposition $\mu = \mu_+-\mu_-$ is minimal, so that $\mu_\pm \leq \mu^\pm$ and $\mu_+ \perp \mu_-$. By monotonicity of the function $H_f$ (see \Cref{concaveDN}), we have $e\geq H(\pm \mu^\pm) \geq H(\pm \mu_\pm)$. Since $H(\mu_+) \perp H(- \mu_-)$, we get
\[e \geq H(\mu_+) + H(-\mu_-) = H(\mu).\]

\medskip
\noindent \emph{Step 5: lower bound for the \(\Gamma-\liminf\).} We justify that \labelcref{mainLowerBound} implies the lower bound \(\Gamma(\Cspace_0')-\liminf_{\eps \to 0} \ener_\eps \geq \mass^{H}\). Indeed, fix $\mu \in \Mspace(\R^N)$ and consider a family $(u_\eps)_{\eps>0}$ weakly converging to $\mu$ as $\eps \to 0$. We need to show that $\mass^H(\mu) \leq \liminf_{\eps \to 0} \ener_\eps(u_\eps)$. Assume without loss of generality that the inferior limit is finite and take a sequence of positive numbers \((\eps_n)_{n\in\N}\to 0\) such that this inferior limit is equal to $\lim_{n\to \infty} \ener_{\eps_n}(u_{\eps_n})$. Now the energy measure \(e_n\) associated with \(u_n=u_{\eps_n}\) has bounded mass and up to extracting a subsequence one may assume that it converges weakly to some measure $e\in \Mspace_+(\R^N)$. By the previous steps, $e \geq H(\mu)$, and by lower semicontinuity and monotonicity of the mass:
\[\liminf_{\eps\to 0^+} \ener_\eps(u_\eps) = \liminf_{n\to\infty} \norm{e_n} \geq \norm{e} \geq \norm{H(\mu)} = \mass^H(\mu).\qedhere\]
\end{proof}

\subsection{Upper bound for the \texorpdfstring{$\Gamma-\limsup$}{Γ-limsup}}

In this section, we introduce the following substitute for \labelcref{hypContinuous}, \labelcref{hypQuasiHomogeneity} and \labelcref{hypGammaCv}, where $f,(f_\eps)_{\eps>0}$ are Borel maps from $\R^N\times \R \to \R^N$ to $[0,+\infty]$:
\begin{enumerate}[(U)]
\item\label{hypUpperGammaLimit} there exists \(C<+\infty\) such that for every \(x,y\in\R^N\), \(u\in\R\) and \(\xi\in\R^N\), 
\[
\limsup_{\eps\to 0^+}f_\eps(x+\eps y,u,\xi)\leq f(x,u,\xi)\quad\text{and}\quad f_\eps(y,u,\xi)\leq C(f(x,u,\xi)+u)\quad\forall\eps>0.
\]
\end{enumerate}

\begin{proposition}\label{limsup_upper_bound}
Assume that $f,(f_\eps)_{\eps > 0}$ satisfy \labelcref{hypUpperGammaLimit} and \labelcref{hypZero}. If \(\mu \in \Mspace(\R^N)\), then there exists \((u_\eps)_{\eps>0} \in W^{1,1}_\loc(\R^N)\) such that $u_\eps \lbm^N \narrowto \mu$ when \(\eps\to 0\) and which satisfies
\[
\limsup_{\eps\to 0^+}\ener_\eps(u_\eps) \leq \mass^{H_{f,\lsc}}(\mu),
\]
where \(H_{f,\lsc}\leq H_f\) stands for the lower semicontinuous envelope of \(H_f\), defined in \labelcref{semicontinuousEnvelope}. In other words, we have \(\Gamma(\Cspace_b')-\limsup_{\eps \to 0} \ener_\eps \leq \mass^{H_{f,\lsc}}\).
\end{proposition}

\begin{proof}[Proof of \Cref{limsup_upper_bound}]
Let \(F=\Gamma(\Cspace_b')-\limsup_{\eps \to 0}\ener_\eps\). As an upper \(\Gamma\)-limit, \(F\) is sequentially lower semicontinuous in the narrow topology. Hence, by \Cref{relaxationMass}, it is enough to prove that \(F(\mu)\leq \mass^{H_f}(\mu)\) whenever \(\mu\) is finitely atomic. Let $\mu = \sum_{i=1}^k m_i \delta_{x_i}$ with \(k\in\N\), \(m_i\in\R\), \(x_i\in\R^N\), and assume without loss of generality that \(x_i\neq x_j\) when \(i\neq j\) and \(\mass^{H_f}(\mu)<+\infty\). Fix \(\eta>0\). For each \(i=1,\dots,k\), there exists \(u_i\in W^{1,1}_\loc(\R^N)\) such that \(\int_{\R^N}u_i=m_i\) and \(\int_{\R^N}f(x_i,u_i,\nabla u_i)\leq H(x_i,m_i)+\eta < +\infty\). We define for every $i=1,\dots,k$,
\begin{gather}
u^i_\eps(x) = \eps^{-N}u_i(\eps^{-1}(x-x_i)), \quad x\in \R^N,
\shortintertext{and}
u_\eps=\max\{u^i_\eps \st i=1,\dots,k\},
\end{gather}
which converge narrowly as measures to $u$ as $\eps \to 0$. We have by change of variables:
\begin{align*}\label{EstimateUpper}
\ener_\eps(u_\eps) &\leq \sum_{i=1}^k \int_{\{u_\eps = u_\eps^i\}} f_\eps(x,\eps^N u_\eps^i(x),\eps^{N+1}\nabla u_\eps^i(x))\eps^{-N} \dd x\\
&\leq \sum_{i=1}^k\ener_\eps(u_\eps^i)=\sum_{i=1}^k\int_{\R^N}f_\eps(x_i+\eps x,u_i,\nabla u_i).
\end{align*}
Using our assumption \labelcref{hypUpperGammaLimit} and the dominated convergence theorem, one gets as \(\eps\to 0\):
\[
F(\mu)\leq\limsup_{\eps\to 0}\ener_\eps(u_\eps)\leq \sum_{i=1}^k\int_{\R^N}f(x_i,u_i,\nabla u_i)
\leq\sum_{i=1}^kH(x_i,m_i)+k\eta=\mass^H(\mu)+k\eta.
\]
The conclusion follows by arbitrariness of \(\eta>0\).
\end{proof}

\subsection{Proof of the main \texorpdfstring{\(\Gamma\)-convergence}{Γ-convergence} result}
We now explain how \Cref{mainGammaConvergence} follows from \Cref{lowerBound} and \Cref{limsup_upper_bound}.

\begin{proof}[Proof of \Cref{mainGammaConvergence}]
The lower bound $\Gamma(\Cspace_0')-\liminf_{\eps \to 0} \ener_\eps \geq \mass^{H_f^-}$ follows from \Cref{lowerBound}, and the upper bound \(\Gamma(\Cspace_b')-\limsup_{\eps \to 0} \ener_\eps \leq \mass^{H_{f,\lsc}}\) from \Cref{limsup_upper_bound}, where the assumption \labelcref{hypUpperGammaLimit} is a consequence of \labelcref{hypContinuous}, \labelcref{hypQuasiHomogeneity} and \labelcref{hypGammaCv}. By \labelcref{hypSlope} and \Cref{concaveDN}, $H_f^- = H_f$, and since $H_f\geq H_{f,\lsc}$ by definition, both $\Gamma-\liminf$ and $\Gamma-\limsup$ (for weak and narrow topologies) coincide with $\mass^{H_f}$.
\end{proof}

\section{Examples, counterexamples and applications}\label{sectionExamples}

\subsection{Scale-invariant Lagrangians and necessity of the slope assumption}\label{sectionCounterExamples}

Our assumption \labelcref{hypSlope} is not very standard, but we need a condition of this type in order to get \(\Gamma\)-convergence of the rescaled energies \(\ener_\eps\) towards \(\mass^{H_f}\), as shown by the following class of scale-invariant Lagrangians:
\begin{equation}
\label{gagliardoExample}
f_\eps(x,u,\xi)= f(u,\xi)
\quad \text{with}\quad 
f(u,\xi)= \begin{dcases*}
u^{p(\frac{1}{p^\star}-1)}\abs{\xi}^p& if $u> 0$,\\
0& else, \end{dcases*}
\end{equation}
where \(p\in (1,N)\), \(N\in\N^\ast\) and $p^\star\coloneqq\frac{pN}{N-p}$. By straightforward computations, \(\ener_\eps(u)= \ener_f(u)\coloneqq \int_{\R^N}f(u,\nabla u)\) for every \(\eps>0\) and \(u\in W^{1,p}_\loc(\R^N)\) in that case.

Moreover, the associated minimal cost function \(H_f\) is not trivial. Indeed, applying the Gagliardo--Nirenberg--Sobolev inequality,
\[
\Big(\int_{\R^N}\abs{v}^{p^\star}\Big)^{\frac{1}{p^\star}}\leq 
C\Big(\int_{\R^N}\abs{\nabla v}^p\Big)^{\frac{1}{p}},\quad\forall v\in L^{p^\star} \cap W^{1,1}_\loc(\R^N),
\]
to the function\footnote{Actually, we apply it to $v_\eps = \phi_\eps(u)$ where $\phi_\eps$ is a suitable approximation of $(\cdot)^{\frac 1{p^\star}}$ and take $\eps\to 0$.} \(v=u^{\frac{1}{p^\star}}\), we obtain that for every \(u\in W^{1,1}_\loc(\R^N,\R_+)\cap L^1(\R^N)\),
\[
\left(\int_{\R^N}u\right)^{\frac{p}{p^\star}}\leq \left(\frac{C}{p^\star}\right)^p \int_{\{u>0\}}u^{\frac{p}{p^\star}-p}\abs{\nabla u}^p=\left(\frac{C}{p^\star}\right)^p\ener_f(u). 
\]
Hence, for every \(m>0\), we have \(H_f(m)>0\), and even \(H_f(m)<+\infty\) since any function \(u=v^{p^\star}\), with \(v\in W^{1,p}(\R^N,\R_+)\), has finite energy. Replacing \(u\) by \(mu\) in the infimum defining \(H_f\) in \labelcref{def_Hf}, we actually obtain
\begin{equation}
\label{gagliardoHmass}
H_f(m)=m^{1-\frac{p}{N}}H_f(1),\quad 0<H_f(1)<+\infty.
\end{equation}
In that case, it is clear that the \(\Gamma\)-limit of \(\ener_\eps\equiv\ener\) in the weak or narrow topology of \(\Mspace_+(\R^N)\), that is the lower semicontinuous relaxation of \(\ener_f\), does not coincide with \(\mass^{H_f}\); indeed, the first functional is finite on diffuse measures whose density has finite energy, while the second functional is always infinite for non-trivial diffuse measures since \(H_f'(0^+)=+\infty\).

These scaling invariant Lagrangians are ruled out by our assumption \labelcref{hypSlope}. All the other assumptions are satisfied except \labelcref{hypCompact}. Note that the following perturbation of \(f\),
\[
\Tilde{f}(u,\xi)=\bigl(1+u^{p(\frac{1}{p^\star}-1)}\bigr)\abs{\xi}^p
\] 
satisfies all the assumptions except \labelcref{hypSlope}, and provides a counterexample to the \(\Gamma\)-convergence. Indeed, \(\mass_{H_{\Tilde{f}}}\geq \mass_{H_f}\) is still infinite on diffuse measures, while (the relaxation of) \(\ener_{\Tilde{f}}\) is finite for any diffuse measure whose density has finite energy.

We stress that an assumption like \labelcref{hypSlope} is actually needed, even for the lower semicontinuity of the function \(H_f\) -- recall that if \(\mass_{H_f}\) is a \(\Gamma\)-limit, then it must be lower semicontinuous by \cite[Proposition~1.28]{braidesGammaConvergenceBeginners2002}, which in turn implies that the function \(H_f\) is lower semicontinuous by \Cref{relaxationMass}. Indeed, consider the Lagrangians
\[
f(x,u,\xi)=\bigl(1+u^{p(\frac{1}{p^\star}-1)}\bigr)\abs{\xi}^{p(x)},
\]
with \(p\in\Cspace^0(\R^N,(1,N))\) such that \(p(0)=p\in (1,N)\) and \(p(x)>p\) when \(x\neq 0\). Then, we have \(H_f(0,m)=m^{1-\frac{p}{N}}H(1)\), but \(H_f(x,\cdot)\equiv 0\) if \(x\neq 0\) as can be easily seen via the change of function \(\eps^N u(\eps\,\cdot)\), with \(\eps>0\) small.

\subsection{General concave costs in dimension one}\label{generalCosts1d}

It has been proved in \cite{wirthPhaseFieldModels2019} that for any continuous concave function \(H:\R_+\to\R_+\) with \(H(0)=0\), there exists a function \(c:\R_+\to\R_+\) such that \(c(0)=0\), \(u\mapsto \frac{c(u)}{u}\) is lower semicontinuous and non-increasing on \((0,+\infty)\), and for every $m \geq 0$,
\[
H(m)=\inf\left\{\int_\R\abs{u'}^2+c(u)\st u\in W^{1,1}_\loc(\R,\R_+),\,\int_\R u=m\right\}.
\]
The Lagrangians of the form \(f_\eps(x,u,\xi)=\abs{\xi}^2+c(u)\), in dimension \(N=1\), satisfy all our assumptions \labelcref{hypLowerContinuous}--\labelcref{hypGammaCv}, hence our \(\Gamma\)-convergence result stated in \Cref{mainGammaConvergence} yields the \(\Gamma\)-convergence of the functionals
\[
\ener_\eps(u)=\int_\R\eps^3\abs{u'}^2+\frac{c(\eps u)}{\eps},\quad u\in W^{1,2}(\R,\R_+),
\]
towards \(\mass^{H}\) for both the weak and narrow convergence of measures. Therefore, we may find an elliptic approximation of any concave $H$-mass. Let us stress that $c$ is determined in \cite{wirthPhaseFieldModels2019} from $H$ through several operations including a deconvolution problem, but no closed form solution is given in general; nonetheless, an explicit solution is provided if $c$ is affine by parts.

In higher dimension \(N\geq 2\), \Cref{propositionStricConcave} tells us that the class of functions \(H=H_f\) with \(f\) satisfying \labelcref{hypLowerContinuous}--\labelcref{hypGammaCv} is smaller, namely, $H$ must satisfy:
\begin{equation}
\label{conditionInverseProblemHigher}
\exists m_\ast\geq 0, \quad
\begin{cases}
H\text{ is linear on }[0,m_\ast],\\
H\text{ is strictly concave }(m_\ast,+\infty).
\end{cases}
\end{equation}
We have no positive or negative answer to the inverse problem, consisting in finding \(f\) satisfying our assumptions such that \(H_f=H\), for a given continuous concave function $H:\R_+\to\R_+$ satisfying \labelcref{conditionInverseProblemHigher}.

\subsection{Homogeneous costs in any dimension}\label{homogeneousCosts}

In this section, we provide Lagrangians $f$ to obtain the $\alpha$-mass $\mass^\alpha \coloneqq \mass^{t\mapsto t^\alpha}$ in any dimension $N$ for a wide range of exponents, including exponents $\alpha \in \left(1-\frac 1N,1\right]$. We consider for every $p \in [1,+\infty),s \in (-\infty,1]$ and $N\in \N^\ast$, the energy defined for every $u\in W^{1,1}_{\loc}(\R^N,\R_+)$ by
\begin{equation}\label{energy_homogeneous_case}
\ener_f(u) \coloneqq \int_{\R^N} f(u,\nabla u) \coloneqq \int_{\R^N} \abs{\nabla u}^p + u^s\mathbf{1}_{u>0}.
\end{equation}
Notice that for $p>1$, $f$ satisfies all our hypotheses \labelcref{hypLowerContinuous}--\labelcref{hypSlope} ; in particular, \labelcref{hypSlope} holds as a consequence of the stronger condition \labelcref{hypSlopeLagrangian} (see \Cref{slopeConditionCorollary}) which is satisfied with $\rho(t) = t$  in dimension $N\geq 2$. Thus by \Cref{mainGammaConvergence} the re-scaled energies $\Gamma$-converge to the $H_{f}$-mass. In this case, we may show that $H_f(m)=cm^\alpha$ for some $\alpha\in(0,1)$, $c\in [0, +\infty]$, and the constant $c$ belongs to $(0,+\infty)$ if and only if $s\in (-p',1)$. Details are given hereafter.

\medskip
\paragraph{\bf Homogeneity of $H_f$.}
In order to compute $H_f$, one may first express $H_f(m)$ as the minimum of a scaling invariant expression by optimizing $\ener_f$ over all mass-invariant rescalings $u_\lambda = \lambda^N v(\lambda \cdot)$ of a given function $u\in W^{1,1}_{\loc}(\R^N,\R_+)$ ; it yields
\begin{multline}\label{rescalingCost}
H_f(m)=\inf\Big\{\inf_{0<\lambda<+\infty}\lambda^{Np+p-N}\int_{\R^N}\abs{\nabla u}^p +\lambda^{Ns-N}\int_{\{u>0\}}u^s\st\\ u\in W^{1,1}_{\loc}(\R^N,\R_+),\,\int_{\R^N}u=m\Big\}.
\end{multline}
Computing the infimum w.r.t. $\lambda$, we obtain
\begin{multline}\label{scaleInvariantCost}
H_f(m)=c_{N,p,s}\inf\Big\{\Big(\int_{\R^N}\abs{\nabla u}^p\Big)^{\alpha_1}\Big(\int_{\{u>0\}}u^s\Big)^{\alpha_2}\st \\u\in W^{1,1}_{\loc}(\R^N,\R_+),\,\int_{\R^N}u=m\Big\},
\end{multline}
with the two exponents
$$
\alpha_1\coloneqq \frac{N(1-s)}{Np+p-Ns},\quad 
\alpha_2\coloneqq \frac{Np+p-N}{Np+p-Ns}
$$
and the constant 
$$
c_{N,p,s} \coloneqq \frac{(Np+p-Ns)(N-Ns)^\frac{N(s-1)}{Np+p-Ns}}{(Np+p-N)^\frac{Np+p-N}{Np+p-Ns}}.
$$
In particular, substituting $u$ with $mu$ in \labelcref{scaleInvariantCost} gives
\begin{equation}
\label{alphaCost}
H_f(m)=m^\alpha H_f(1),\quad\text{with} \quad\alpha\coloneqq p\alpha_1+s\alpha_2=\frac{Np+sp-Ns}{Np+p-Ns}.
\end{equation}
It remains to ensure that this function is not trivial, i.e. $0<H_f(1)<+\infty$.

\medskip
\paragraph{\bf Upper bound: $H_f(1)<+\infty$.}
In the case $s\in [0,1]$, any $u\in\Cspace_c^1(\R^N)$ has finite energy, thus $H_f$ is non-trivial for every $p\in [1,+\infty)$. In the case $s< 0$, consider the competitor $u :x \mapsto (1-\abs{x})_+^\gamma$ for $\gamma > 0$ to be fixed later. Then $\int_{\R^N} \abs{\nabla u}^p < +\infty$ if and only if $t \mapsto (1-t)^{(\gamma-1)p}$ is integrable at $1^-$, that is if $(\gamma-1)p > -1$ or, equivalently, $ \gamma > 1- 1/p$. Similarly, $\int_{\{u>0\}} u^s < +\infty$ if and only if $\gamma s > -1$ or, equivalently, $\gamma < -1/s$. Therefore, one may find $\gamma > 0$ satisfying both conditions, and ensure that $H_{f}$ is non-trivial, if
\[-p'<s.\]

\medskip
\paragraph{\bf Lower bound: $H_f(1)>0$.} Here, we can assume w.l.o.g. that $H_f(1)<+\infty$. In the case $s=1$, taking any competitor $u$ is with finite energy and $\lambda\to 0^+$ in \labelcref{rescalingCost} yields $H_f(m)=m$ for every $m\geq 0$. Hence, we may also assume that $s<1$. Note that in view of \labelcref{scaleInvariantCost} and \labelcref{alphaCost}, we know that $H_f(1)$ is related to the greatest constant $c\geq 0$ such that the Gagliardo-Nirenberg-Sobolev (GNS) type inequality\footnote{The notable difference with the classical GNS inequality is that the exponent $s$ is smaller than $1$, and may even be negative.}
\begin{equation}\label{GNS_type_inequality}
c\Big(\int_{\R^N}u\Big)\leq\Big(\int_{\R^N}\abs{\nabla u}^p\Big)^\frac{Np(1-s)}{Np+sp-Ns}\Big(\int_{\{u>0\}}u^s\Big)^\frac{Np+p-N}{Np+sp-Ns},\quad \forall u\in W^{1,1}_{\loc}(\R^N)
\end{equation}
holds true, namely, we have $c=\Big(\frac{H_f(1)}{c_{N,p,s}}\Big)^{1/\alpha}$; hence, the inequality \labelcref{GNS_type_inequality} holds with $c>0$ if and only if $H_f(1)>0$. This is actually equivalent to proving existence of a minimizer of $\ener_f$ over functions $u\in W^{1,1}_{\loc}(\R^N,\R_+)$ with $\int_{\R^N}u=1$ in view of \Cref{existenceProfile}. We prove this, together with uniqueness and basic properties of optimal profiles in the next paragraph.

\medskip
\paragraph{\bf Existence, uniqueness and properties of optimal profiles.} It is well-known that optimal profiles in the classical Gagliardo-Nirenberg-Sobolev inequalities do exist ; besides, up to rescalings and translations they are unique, radially decreasing and compactly supported. Existence of radially decreasing solutions is a consequence of P\'olya–Szeg\"o inequality (see for instance \cite{brothersMinimalRearrangementsSobolev1988}), compactness of the support follows for example from the compact support principle of Pucci, Serrin and Zou \cite{pucciStrongMaximumPrinciple1999} or Pohozaev-type identities, and uniqueness from the work of Serrin and Tang \cite{serrinUniquenessGroundStates2000} for instance. All these techniques may be adapted to our case for exponents $s\in [0,1)$, but we did not find a comprehensive reference, even more so when $s \in (-p',0)$, except for $p=2$ and $s<0$ which is treated in  \cite{dubsProblemesPerturbationsSingulieres1998}. For this reason and for self-containedness, we provide a sketch of proof for the existence, symmetry, compactness of the support for $s\in (-p',1])$ of optimal profiles in \labelcref{energy_homogeneous_case}. We are only able to justify uniqueness for $s\in (0,1]$.

\smallskip

{\it 1. Existence of an optimal profile which is radially symmetric.} The radially symmetric decreasing rearrangement $u^*$ of an admissible function $u\in W^{1,1}_{\loc}(\R^N,\R_+)$ satisfies $\norm{\nabla u^*}_p\leq\norm{\nabla u}_p$ by the P\'olya–Szeg\" o inequality, and we have $\int_{\R^N}\abs{u^*}^s\leq\int_{\R^N}\abs{u}^s$ by equimeasurability. Hence, 
$$\ener_f(u^*)\leq\ener_f(u).$$
Thus we can restrict the minimization to radially symmetric non-increasing functions. Take a minimizing sequence $(u_n)_n$ in this class: $u_n(\cdot)=v_n(\abs{\cdot})$ with $v_n:[0,+\infty)$ non-increasing, $\int_{\R^N}u_n=m\in\R_+$, and $\ener_f(u_n)\to H_f(m)$. $(u_n)_n$ is weakly precompact in $L^1_\mathrm{loc}(\R^N)$, by compact Sobolev embedding, and even globally precompact in $L^1(\R^N)$. Indeed, when $s > 0$, the upper bound 
    $$\abs{\mathbb{S}^{N-1}}v_n(r)^s \frac{r^N}N\leq \abs{\mathbb{S}^{N-1}}\int_0^r v_n(t)^s t^{N-1}\dd t\leq\ener_f(u_n)\leq C<+\infty,\quad\forall r\in\R_+,$$
    gives a uniform integrable decay in $\abs{x}^{-N/s}$ at infinity ; when $s\leq 0$, we use the elementary inequality $t+t^s\geq 1$ for all $t>0$, to obtain
    \begin{equation}
    \label{compatnessSupport}
\abs{\{u_n>0\}}\leq \int_{\{u_n>0\}} (u_n+ u_n^s) \dd x\leq m+C<+\infty,
    \end{equation}
    so that the size of the support of $u_n$ (a ball of radius $R_n>0$) is uniformly bounded w.r.t. $n$. 
    
    We deduce by lower semicontinuity of $\ener_f$, that, extracting a subsequence if necessary, $(u_n)_n$ converges in $L^1$ to a minimizer $u$ of $\ener_f$ with mass $m$, i.e. $\ener_f(u)=H_f(m)$ and $\int_{\R^N}u=m$. In particular, $$H_f(1)>0.$$

{\it 2. Euler-Lagrange equation and compactness of the support.} Let $u(\cdot)=v(\abs{\cdot})$ be a global minimizer of \labelcref{energy_homogeneous_case} with mass $m>0$, with $v:\R_+\to\R_+$ non-increasing. Let also $\Bar B(0,R)$ be the support of $u$, with $0<R\leq +\infty$. (Here, $\Bar B(0,R)=\R^N$ if $R=+\infty$.) Computing the first order variation of the energy, we obtain that for every test function $w\in\mathcal{C}_c^1(\R^N)$ which is compactly supported in $B(0,R)$ and satisfies $\int_{\R^N}w(x)\dd x=0$, we have
    \begin{equation}
    \label{firstOrderConditionU}
    \langle \delta \ener_f(u)\,,\, w \rangle \coloneqq \int_{\R^N} (p(\nabla u)^{p-1}\cdot\nabla w+ su^{s-1}w)\dd x= 0.
        \end{equation}
In other words, $u$ solves in the weak sense the Euler-Lagrange equation
$$
-p\Delta_p u+su^{s-1}=\lambda,  \quad\text{in }B(0,R),
$$
where $\lambda\in\R$ is the Lagrange multiplier associated to the mass constraint. One can see that 
$$\lambda=H_f'(m)=\alpha m^{\alpha-1}H_f(1).$$ 
Indeed, let $u_1$ be a minimizer of $\ener_f$ of mass $1$ ; then a minimizer of $\ener_f$ with mass $m$ is given by $u_m(\cdot)\coloneqq m\lambda_m^N u_1(\lambda_m\,\cdot)$, with $\lambda_m=m^\frac{s-p}{Np+p-Ns}$ (see \labelcref{rescalingCost}) ; hence,
    \begin{multline}
    H_f'(m)=\frac{\dd}{\dd m}\ener_f(u_m
   )=\langle\delta\ener_f(u_m)\,,\, \frac{\dd}{\dd m}u_m \rangle_{L^2(\R^N)}\\ =\langle\lambda\,,\, \frac{\dd}{\dd m}u_m \rangle_{L^2(\R^N)}=\frac{\dd}{\dd m}\langle\lambda\,,\, u_m \rangle_{L^2(\R^N)}=\frac{\dd}{\dd m}(\lambda m) =\lambda.
    \end{multline}
In particular, $\lambda$ is unique.  

From the Euler-Lagrange equation, we also get that $u$ is smooth in $B(0,R)$ by a bootstrap argument ; hence, $u$ solves the Euler-Lagrange equation in the classical sense. In terms of the profile $v$, the Euler-Lagrange equation rewrites
\begin{equation}
\label{EulerLagrange}
-(pv'(r)^{p-1}r^{N-1})'+sv(r)^{s-1}r^{N-1}=\lambda r^{N-1},\quad\forall r\in(0,R).
\end{equation}
For every $r\in (0,R)$, integrating \labelcref{EulerLagrange} on $(0,r)$ yields
\begin{equation}
\label{integrationODE}
pv'(r)^{p-1}r^{N-1}=\int_0^r(sv(\rho)^{s-1}-\lambda)\rho^{N-1}\dd \rho.
\end{equation}

The LHS in \labelcref{integrationODE} is non-positive as $v'\leq 0$ in $(0,R)$. If $s > 0$ and $R=+\infty,$ the RHS is $+\infty$ since the integrand goes to $+\infty$ as $\rho\to R$,  which is a contradiction. When $s\leq 0$, we already saw in \eqref{compatnessSupport} that the support of $v$ is bounded. In any case, we have thus proved that
$$
R<+\infty.
$$

{\it 3. Uniqueness when $s>0$.} We justify that there is a unique minimizer of the energy with mass $m$. Once we know that $v'$ does not vanish on $(0,R)$, the case of equality in P\' olya–Szeg\"o inequality (see \cite[Theorem~1.1]{brothersMinimalRearrangementsSobolev1988}) implies that any minimizer $u$ satisfies $u=u^*$. Then, applying \cite[Theorem~1]{serrinUniquenessGroundStates2000} yields the uniqueness of radial solution to $\Delta_p u+f(u)=0$ for the non-linearity $f(t)=\lambda-st^{s-1}$, thus we get that $v$ is unique. 

 We now prove that $v'<0$ on $(0,R)$.  Since the LHS in \eqref{integrationODE} is non-positive, the non-decreasing function $g:\rho\mapsto sv(\rho)^{s-1}-\lambda$, which tends to $+\infty$ as $\rho\to R$, must be negative near $0$. If we had $v'(r)=0$ for some $r\in(0,R)$, then the non-positive function $t\mapsto pv'(t)^{p-1}t^{N-1}$ would be maximal and its derivative would vanish at $t=r$. By \eqref{integrationODE}, this means that $g(r)=sv(r)^{s-1}-\lambda=0$ so that $g\leq 0$ on $[0,r]$, but also that $g(\rho)\rho^{N-1}$ integrates to $0$ on $[0,r]$. Hence, $g\equiv 0$ on $[0,r]$, a contradiction.

\medskip
\paragraph{\bf Conclusion and range of $\alpha$-masses obtained in this way.} To summarize, we have shown that if \(-p'<s\leq 1\) then $H_{f}$ is non-trivial. The converse is true. Indeed, if $H_f(1)<+\infty$ then, by the preceding, there exists a radial decreasing minimizer $u(\cdot)=v(\abs{\cdot})$ with $v:[0,+\infty)\to\R_+$ non-increasing and compactly supported on $[0,R]$, $0<R<+\infty$. But, by the Young inequality, $\ener_f(u)\geq c\int_{R/2}^R v(r)^{s/p'}\abs{v'(r)}r^{N-1}\dd r\geq c_R\int_0^{v(R/2)}t^{s/p'}\dd t$ with $c_R>0$, the latter being finite if and only if $s/p'>-1$. 

Since $\alpha$, in \labelcref{alphaCost}, is a monotone function of $s$, one may easily compute the range of parameters $\alpha$ that we obtain. If $p$ and $N$ are fixed, $\alpha$ ranges over $\left(\frac{N-1}{N+1-1/p}, 1\right]$ when $s\in (-p',1]$. Hence, when $N=1$ we obtain the whole range $\alpha \in (0,1]$, and at least the range $\alpha \in \left(1-\frac{2}{N+1},1\right]$ when $p$ ranges over $(1+\infty)$ in dimension $N\geq 2$.

%%%%%%%%%%
%%%%%%%%%%%

\subsection{Branched transport approximation: \texorpdfstring{$H$-masses}{H-masses} of normal \texorpdfstring{$1$-currents}{1-currents}}\label{branchedTransport}

Branched Transport is a variant of classical optimal transport (see \cite{santambrogioOptimalTransportApplied2015} and Section~4.4.2 therein for a brief presentation of branched transport, and \cite{bernotOptimalTransportationNetworks2009} for a vast exposition) where the transport energy concentrates on a network, i.e. a \(1\)-dimensional subset of \(\R^d\), which has a graph structure when optimized with prescribed source and target measures. It can be formulated as a minimal flow problem, 
\[
\min\Big\{\mass_1^H(w)\st \mathrm{div}(w)=\mu^--\mu^+\Big\},
\]
where \(\mu^\pm\) are probability measures on \(\R^d\), \(H:\R^d\times\R_+\to\R_+\) is mass-concave, and the \(H\)-mass \(\mass_1^H\) is this time defined for finite vector measures $w \in \Mspace(\R^d,\R^d)$ whose distributional divergence is also a finite measure; in the language of currents, it is called a $1$-dimensional normal current. Any such measure may be decomposed into a $1$-rectifiable part $\theta\xi \cdot \hdm^1 \mres \Sigma$ where \(\theta(x)\geq 0\) and \(\xi(x)\) is a unit tangent vector to \(\Sigma\) for \(\hdm^1\)-a.e. \(x\in \Sigma\), and a $1$-diffuse part $w^\perp$ satisfying $\abs{w^\perp}(A) = 0$ for every $1$-rectifiable set $A$:
\[
w=\theta\xi\cdot\hdm^1 \mres M +w^\perp.
\]
The $H$-mass is then defined by:
\begin{equation}
\label{HmassVector}
\mass_1^H(w)\coloneqq\int_\Sigma H(x,\theta(x))\dd\hdm^1(x)+\int_{\R^d} H'(x,0)\dd\abs{w^\perp}.
\end{equation}

In the case \(H(x,m) = m^\alpha\) with \(0<\alpha<1\), a family of approximations of these functional has been introduced in \cite{oudetModicaMortolaApproximationBranched2011}:
\begin{equation}
\ener_{\eps}(w)=
\begin{cases}
\int_{\R^d}\eps^{\gamma_1}\abs{\nabla v}^2+\eps^{-\gamma_2}\abs{v}^\beta&\text{if $w = v\lbm^d$, \(v\in W^{1,2}_\loc(\R^d,\R^d)\),}\\
+\infty&\text{otherwise,}
\end{cases}
\end{equation}
with \(\beta=\frac{2-2d+2\alpha d}{3-d+\alpha(d-1)}\), \(\gamma_1=(d-1)(1-\alpha)\) and \(\gamma_2=3-d+\alpha (d-1)\).  It has been shown in \cite{oudetModicaMortolaApproximationBranched2011,monteilUniformEstimatesModica2017} that the functionals \(\ener_{\eps}\) \(\Gamma\)-converge as \(\eps\to 0^+\), in the topology of weak convergence of \(u\) and its divergence, to a non-trivial multiple of the \(\alpha\)-mass \(\mass_1^\alpha \coloneqq \mass_1^H\) with \(H(x,m)=m^\alpha\) in dimension \(d=2\). The result extends to any dimension \(d\), by \cite{monteilEllipticApproximationsSingular2015}, thanks to a slicing method that relates the energy \(\ener_{\eps}\) with the energy of the sliced measures \(u=(w\cdot\nu)_+\) supported on the slices \(V_a=\{x\in\R^d\st x\cdot \nu=a\}\simeq\R^N\), for any given unit vector \(\nu\in\R^d\), defined by
\[
\bar \ener_{\eps}(u) = \int_{\R^N}\eps^{\gamma_1}\abs{\nabla u}^2+\eps^{-\gamma_2}\abs{u}^\beta.
\]
The functionals \(\bar\ener_{\eps}\) \(\Gamma\)-converge as \(\eps\to 0^+\), in the narrow topology, to \(c\mass^\alpha\) for some non-trivial $c$, as shown in \Cref{homogeneousCosts}, and one may recover every $\alpha$-mass in this way for $\alpha \in \left(\frac{2d-4}{2d+1},1\right]$, and in particular every so-called super-critical exponents for Branched Transport in dimension $d$, that is $\alpha \in (1-1/d,1]$.

The same slicing method would allow to extend our \(\Gamma\)-convergence result stated in \Cref{mainGammaConvergence} to functionals defined on vector measure
\begin{equation}
\ener_\eps(w)=
\begin{cases}
\int_{\R^d}f_\eps(x,\eps^{d-1} \abs{v}(x), \eps^{d}\abs{\nabla v}(x)) \eps^{1-d} \dd x &\text{if \(w = v\lbm^d, v\in W^{1,1}_\loc(\R^d,\R^d)\),}\\
+\infty&\text{otherwise,}
\end{cases}
\end{equation}
for Lagrangians \(f_\eps\to f\) fitting the framework of \Cref{mainGammaConvergence}. The expected \(\Gamma\)-limit, for the weak topology of measures and their divergence measure, would be the functional \(\mass_1^{H_f}\), with \(H_f\) defined in \labelcref{def_Hf}. Note that this approach would provide approximations of \(H\)-masses for more general continuous and concave cost functions \(H:\R_+\to\R_+\) satisfying \(H(0)=0\). By \cite{wirthPhaseFieldModels2019}, we would obtain all such \(H\)-masses when \(N=1\) (corresponding to \(d=2\)). 

\subsection{A Cahn-Hilliard model for droplets\label{droplets}}
Following the works \cite{bouchitteTransitionsPhasesAvec1996} in the one-dimensional case and \cite{dubsProblemesPerturbationsSingulieres1998} in higher dimension, we consider functionals on $\Mspace_+(\R^N)$ of the form:
\begin{equation}\label{defW}
\mathcal W_\eps(\mu) = \begin{dcases*}
\int_{\R^N} \eps^{-\rho} (W(u) + \eps\abs{\nabla u}^2) & if $\mu = u \lbm^N, u\in W^{1,1}_\loc(\R^N,\R_+)$,\\
+\infty& otherwise,
\end{dcases*}
\end{equation}
where $W : \R_+ \to \R_+$ is a Borel function satisfying $W(t) \sim_{u\to+\infty} u^s$ for some exponent $s\in (-\infty,1)$. In \cite{bouchitteTransitionsPhasesAvec1996, dubsProblemesPerturbationsSingulieres1998}, it is in particular proven, under some assumptions on the slope of $W$ at $0$ and its regularity, that the family $(\mathcal W_\eps)_{\eps > 0}$ $\Gamma$-converges to a non-trivial multiple of the $\alpha$-mass, $\alpha = \frac{1-s/2+s/N}{1-s/2+1/N}$, when $s \in (-2,1)$ and $\rho = \rho(s,N) \coloneqq \frac{N(1-s)}{(N+2)+N(1-s)}$.
In this section, we recover this $\Gamma$-convergence result using our general model.

Replacing $\eps$ with $\bar \eps \coloneqq \eps^{(N+2)+N(1-s)}$ and noticing that $1-\rho = \frac{N+2}{(N+2)+N(1-s)}$, one gets for every $u\in W^{1,1}_\loc(\R^N,\R_+)$:
\begin{align*}
\mathcal W_{\bar \eps}(u)&= \int_{\R^N} \eps^{-N(1-s)} W(u) + \eps^{N+2} \abs{\nabla u}^2\\ &= \int_{\R^N} \left( [\eps^{Ns} W(\eps^{-N} \eps^N u)] + \abs{\eps^{N+1} \nabla u}^2 \right) \eps^{-N}\\
&= \int_{\R^N} f_\eps^W(x,\eps^N u,\eps^{N+1} \nabla u) \eps^{-N},
\end{align*}
where $f_\eps^W$ is defined for every $x\in \R^N,u\in \R_+, \xi \in \R^N$ by
\begin{equation*}
f_\eps^W(x,u,\xi) \coloneqq W_\eps(u) + \abs{\xi}^2 \quad \text{and}\quad W_\eps(u) \coloneqq \eps^{Ns} W(\eps^{-N}u).
\end{equation*}
Therefore if we take $f_\eps = f_\eps^W$ in our general model \labelcref{mainEnergyEpsilon} we exactly get $\mathcal W_{\bar\eps} = \ener_\eps$. The fact that $W(u) \sim u^s$ as $u \to +\infty$ implies that $W_\eps$ converges pointwise to the map $w_s : u \mapsto u^s$ if $u > 0$, $w_s(0) = 0$, hence $f_\eps^W$ converges to $f_s : (x,u,\xi) \mapsto w_s(u) + \abs{\xi}^2$.
\begin{theorem}\label{GammaConvergenceW}
Assume that $W : \R_+ \to \R_+$ satisfies:
\begin{enumerate}[\textnormal{(HW}$_1$\textnormal{)},leftmargin=*]
\item\label{hypWlsc} $W$ is lower semicontinuous,
\item\label{hypWzero} $\{W = 0 \} = \{ 0\}$,
\item\label{hypWpower} $W(u) \sim_{u\to +\infty} u^s$ for some $s\in (-\infty,1)$,
\item\label{hypWbound} $\displaystyle\sup_{u>0} \frac{W(u)}{u^s} < +\infty$,
\item\label{hypWslope} $\displaystyle 0 < \liminf_{u\to 0^+}\frac{W(u)}u$.
\end{enumerate}
Then $(\mathcal W_\eps)_{\eps > 0}$ $\Gamma$-converges to $\mass^{H_{f_s}}$, for both topologies $\Cspace_0'$ and $\Cspace_b'$, and if $s \in (-2,1]$ then $\mass^{H_{f_s}}$ is a non-trivial multiple of $\mass^\alpha$ where $\alpha = \frac{1-s/2+s/N}{1-s/2+1/N}$.
\end{theorem}

To prove this theorem, we start with a simple lemma.

\begin{lemma}\label{lemmaW}
Assume that $W$ satisfies \labelcref{hypWlsc}--\labelcref{hypWslope}. Then for every $\delta \in (0,1)$, there exists $c_\delta \in (0,+\infty)$ such that for every $\eps > 0$ and every $u \in \R_+$,
\begin{equation}
\delta ( u^s \wedge c_\delta \eps^{-N(1-s)} u) \leq W_\eps(u).
\end{equation}
\end{lemma}
\begin{proof}
Fix $\delta \in (0,1)$. There exists $M > 0$ such that $\delta u^s \leq W(u)$ for every $u \geq M$. Besides, the map $w : u \mapsto W(u)/u$ is lower semicontinuous and positive on $(0,M]$ by \labelcref{hypWlsc} and \labelcref{hypWzero}, and since $\liminf_{u\to 0} w(u) > 0$ by \labelcref{hypWslope}, $w$ is necessarily bounded from below on $(0,M]$ by some constant $c > 0$. As a consequence $W_\eps(u) \geq \delta u^s$ if $u\geq \eps^N M$ and $W_\eps(u) \geq c \eps^{N(s-1)}u$ if $u\leq \eps^N M$, hence:
\[\forall u\in \R, \quad W_\eps(u) \geq \delta ( u^s \wedge c \eps^{-N(1-s)} u).\qedhere\]
\end{proof}

\begin{proof}[Proof of \Cref{GammaConvergenceW}]
By \labelcref{hypWbound}, there exists a constant $C$ such that $f_\eps^W \leq C f_s$ for every $\eps$, and since $f_\eps^W$ does not depend on the $x$ variable and converges pointwise to $f_s$, \labelcref{hypUpperGammaLimit} is satisfied and our $\Gamma-\limsup$ result stated in \Cref{limsup_upper_bound} yields
\[\mass^{H_{f_s}} \geq \Gamma(\Cspace_b')-\limsup_{\eps \to 0} \ener_\eps.\]
Fix $\delta \in (0,1)$. By \Cref{lemmaW}, there exists $c_\delta$ such that
\[\forall x,u,\xi, \quad f_\eps^W(x,u,\xi) \geq \delta(\abs{\xi}^2 + (u^s \wedge c_\delta \eps^{-N(1-s)} u) \eqqcolon f_\eps^\delta(x,u,\xi).\]
It is easy to check that $f^\delta_\eps$ satisfies \labelcref{hypLowerContinuous}, \labelcref{hypConvex} and \labelcref{hypCompact} for every $\eps > 0$. Moreover $f_\eps^\delta \uparrow \delta f_s$ and $(f_{\eps}^\delta)_-'(\cdot,0^+,0) = \delta c_\delta \eps^{-N(1-s)} \uparrow (+\infty) = (\delta f_s)_-'(\cdot,0^+,0)$ as $\eps \to 0$, thus \labelcref{hypGammaCv} holds for the family $(f_\eps^\delta)_{\eps > 0}$, and by applying our $\Gamma-\liminf$ result stated in \Cref{lowerBound} to the energies $\ener_\eps^\delta$ induced by $f_\eps^\delta$ we get:
\[\Gamma(\Cspace_0')-\liminf \ener_\eps \geq \Gamma(\Cspace_0')-\liminf \ener_\eps^\delta \geq \mass^{H_{\delta f_s}^-}.\]
We get the result by taking the limit $\delta \to 1$, noticing that $(f_s)_-'(\cdot,0^+,0) = +\infty$, so that $H_{\delta f_s}^- = H_{\delta f_s} = \delta H_{f_s}$ and $\mass^{H_{\delta f_s}^-} = \mass^{\delta H_{f_s}} = \delta \mass^{H_{f_s}}$.
\end{proof}

\begin{remark}
We recover the $\Gamma$-convergence results of \cite{bouchitteTransitionsPhasesAvec1996} and \cite{dubsProblemesPerturbationsSingulieres1998} when $s \in (-2,1)$ under slightly more general assumptions: besides \labelcref{hypWzero} and \labelcref{hypWpower}, the authors impose the existence of a non-trivial slope $\lim_{u\to 0} \frac{W(u)}u \in (0,+\infty)$ and a regularity condition (either $W$ is of class $\Cspace^1$ or continuous and non-decreasing close to $0$), which are stronger than \labelcref{hypWlsc}, \labelcref{hypWbound} and \labelcref{hypWslope}. Let us stress however that these works also tackle the cases $s < -2$ in any dimension, where the exponent $\rho$ has to be fixed to $\rho(-2,N)$, and the case $s=-2$ in dimension one, where a logarithmic factor must be introduced, replacing $\eps^{-\rho}$ with $\eps^{-\rho(-2,1)} \abs{\log\eps}^{-1} = \eps^{-1/2} \abs{\log \eps}^{-1}$. This implies that in our model we get a trivial $\Gamma$-limit when $s \leq -2$, namely $H_{f_s} \equiv +\infty$ on $(0,+\infty)$.
\end{remark}

\printbibliography

\end{document}